\documentclass[10pt,a4paper]{amsart}

\usepackage{amsthm,amsmath,amssymb,amsxtra,amscd,enumerate}
\usepackage[hypertex]{hyperref}% 
\usepackage{longtable}
\usepackage{float}

\numberwithin{equation}{section}

\theoremstyle{plain}
\newtheorem{thm}{Theorem}[section]
\newtheorem{de}[thm]{Definition}
\newtheorem{prop}[thm]{Proposition}
\newtheorem{lem}[thm]{Lemma}
\newtheorem{cor}[thm]{Corollary}
\newtheorem{rem}[thm]{Remark}
\newtheorem{fact}[thm]{Fact}

\newtheorem{thmalph}{Theorem}

\newtheorem{probalph}[thmalph]{Problem}

\newcommand{\bb}{\mathbb}
\newcommand{\cal}{\mathcal}
\newcommand{\op}{\operatorname}

\title[symmetric pairs with discretely decomposable modules]
{Classification of symmetric pairs with discretely decomposable restrictions
of $(\frak{g},K)$-modules}

\thanks{%
}

\author[Toshiyuki Kobayashi]{Toshiyuki Kobayashi$^*$}
\address[Toshiyuki Kobayashi]{Kavli IPMU
         and Graduate School of Mathematical Sciences, 
         The University of Tokyo, 3-8-1 Komaba, Meguro, 
        153-8914 Tokyo, Japan}
\thanks{* Partially supported by
        Grant-in-Aid for Scientific Research (B) (22340026), Japan
        Society for the Promotion of Science}
\author[Yoshiki Oshima]{Yoshiki Oshima$^{**}$}
\thanks{** Supported by
        Grant-in-Aid for JSPS Fellows}
\address[Yoshiki Oshima]{Kavli IPMU, The University of Tokyo, 
           5-1-5 Kashiwanoha, Kashiwa, 277-8583 Chiba, Japan}

\begin{document}

\maketitle
\begin{abstract}
We give a complete classification
 of reductive symmetric pairs $(\frak{g}, \frak{h})$
with the following property:
there exists at least one infinite-dimensional irreducible $(\frak{g},K)$-module $X$
that is discretely decomposable
as an $(\frak {h},H\cap K)$-module.  
% in the sense of [T.~Kobayashi, Invent.~Math.~94].  

We investigate further
 if such $X$ can be taken to be a minimal representation,
 a Zuckerman derived functor module $A_{\frak{q}}(\lambda)$, 
 or some other unitarizable $(\frak{g},K)$-module.  
The tensor product 
 $\pi_1 \otimes \pi_2$
 of two infinite-dimensional irreducible $(\frak{g},K)$-modules
 arises as a very special case
 of our setting.  
In this case,
 we prove that $\pi_1 \otimes \pi_2$ 
 is discretely decomposable if and only if
 they are simultaneously highest weight modules.

\end{abstract}

\section{Introduction}
The subject of this article is discretely decomposable restrictions 
 of irreducible representations
 with respect to symmetric pairs.  

In order to explain our motivation,
 we begin by confining ourselves 
 to unitary representations.  
Let $\pi$ be an irreducible unitary representation
 of a Lie group $G$, 
 and $H$ a subgroup in $G$.  
We may think of $\pi$
 as a representation 
 of the subgroup $H$,
 denoted simply by $\pi|_{H}$.  
Then the restriction $\pi|_{H}$
 is no longer irreducible in general,
 but is unitarily equivalent to a direct integral 
 of irreducible unitary representations of $H$, 
 possibly with continuous spectrum.  
Branching problems ask how the restriction 
 $\pi|_{H}$ decomposes.  

In the case where $(G,H)$ is a pair of real reductive Lie groups,
 we can take $K$ and $H\cap K$ to be
 maximal compact subgroups of $G$ and $H$, 
 respectively. 
Then, 
 as an algebraic analogue of branching problems 
 of unitary representations,
 we may consider 
 how the underlying $(\frak{g},K)$-module $X$ of $\pi$
 behaves as an $(\frak{h}, H\cap K)$-module
 in the category of Harish-Chandra modules.  
We proved in \cite{kob98ii}
 that either (1) occurs or (2) occurs:
\begin{itemize}
\item[(1)]
 $X$ is discretely decomposable
 as an $(\frak{h}, H\cap K)$-module.  
\item[(2)]
$\operatorname{Hom}_{\frak{h}, H\cap K}(Y,X)=0$
 for any irreducible $(\frak{h}, H\cap K)$-module $Y$.  
\end{itemize}
The case (1) fits well into algebraic approach 
 to branching problems.  
In this case, the branching laws of the restrictions of $\pi$ and $X$
 coincide in the following sense:
\begin{alignat*}{3}
&\pi|_{H}
&&\simeq \sideset{}{^\oplus}\sum_{\tau\in\widehat{H}}
         m_\pi(\tau)\tau
&&\quad\text{(Hilbert direct sum)},
\\
&X|_{(\mathfrak{h}, H\cap K)}
&&\simeq \bigoplus_{\tau\in\widehat{H}}
         m_\pi(\tau)\tau_{H\cap K}
&&\quad\text{(algebraic direct sum)},
\end{alignat*}
where $\widehat{H}$ denotes the set of equivalence classes of irreducible
unitary representations of $H$, 
 $\tau_{H\cap K}$ is 
 the underlying $(\frak{h}, H\cap K)$-module
 of $\tau$,
 and the multiplicity $m_{\pi}(\tau)$
 is the same in both the analytic and algebraic
 branching laws \cite{ASPM}.    
On the other hand,
 the case (2) occurs if the irreducible decomposition 
 of the restriction $\pi|_{H}$ involves continuous spectrum.

More generally,
 the feature of the restriction of irreducible $(\frak{g},K)$-modules
 $X$
 remains essentially 
 the same as  $(\frak{h}, H\cap K)$-modules without assuming unitarizability.
 Namely, 
 either (1) or (2) occurs 
 for any irreducible $(\frak{g},K)$-module $X$
even if $X$ does not come from a unitary representation
 of the group $G$
(we note that \lq{discrete decomposability}\rq\
 in Definition~\ref{beta}
 is slightly weaker
 than \lq{complete reducibility}\rq).  

Obviously (1) always holds
 for the pairs $(G,H)$ with $H=K$ because $K$ is compact,
 whereas (1) never holds
 for the pair $(G,H)=(SL(n,{\mathbb{C}}),SL(n,{\mathbb{R}}))$
 if $\dim X=\infty$ (\cite[Theorem 8.1]{kob98ii}).  
Such pairs are examples of so-called {\it{symmetric pairs}} $(G,G^{\sigma})$, 
 where $G^{\sigma}$ is the fixed point group
 of an involutive automorphism $\sigma$
 of $G$.    
The classification of reductive symmetric pairs
 was accomplished
 by M.~Berger \cite{ber}
 on the Lie algebra level
 $(\frak{g},\frak{g}^{\sigma})$.

In this paper
 we highlight the restriction
 of representations
 with respect to symmetric pairs $(G,G^{\sigma})$.  
The tensor product
 of two representations
 can be treated
 as a special case
 of this framework.  
Indeed, 
 the \lq{group case}\rq\
 $(G' \times G',
 \operatorname{diag}G')$
 is a symmetric pair
 as $\operatorname{diag} G'$
 is the fixed point group
 of the involution $\sigma$
 given by $\sigma(x,y)=(y,x)$.
Thus branching laws
 with respect to symmetric pairs
 are thought of as a natural generalization
of the irreducible decomposition 
 of the tensor product representations.

We consider the following.  

\begin{probalph}
\label{prob:a}
Classify all the reductive symmetric pairs
 $(G,G^\sigma)$
 for which there exists at least one infinite-dimensional 
 irreducible $(\frak{g},K)$-module $X$
 satisfying the property (1).  
\end{probalph}

The problem reduces to the following two cases: 
\begin{itemize}
\item $\frak {g}$ is a simple Lie algebra;
\item $(G,G^{\sigma})$ is the \lq{group case}\rq\
 $(G' \times G',\operatorname{diag} G')$
 with $\frak {g}'$ simple.  
\end{itemize}

Our main result of this paper
 is a complete solution to 
 Problem~\ref{prob:a} 
 on the Lie algebra level.  
The classification is given in
 Theorem~\ref{thm:ddpair}
 for simple $\frak g$, 
 and in Theorem~\ref{tensor}
 for the \lq{group case}\rq.
For simple $\frak g$, 
 we shall see
 that there is quite a rich family
 of such symmetric pairs $(G,G^{\sigma})$
 in addition to the obvious case
 where $G^{\sigma}=K$ 
 or where $(\frak {g},\frak{g}^{\sigma})$ is 
 of holomorphic type
 (Definition~\ref{de:holpair}).  
See Table~\ref{ddpairs}.  
Our list contains even the case 
 where $\frak{g}$ is complex and
 $\frak{g}^\sigma$ is its real form (Corollary~\ref{complex}).

In the course of the proof, 
 we need a finer understanding of the $(\frak{g}, K)$-modules $X$
 that are discretely decomposable
 as $(\frak {g}^{\sigma}, K^{\sigma})$-modules
 (cf.~\cite{kob94,kob98i,kob98ii}):
\begin{probalph}
\label{prob:b}
Let $(\frak {g}, \frak{g}^{\sigma})$ be 
 a reductive symmetric pair.  
Which infinite-dimensional irreducible $(\frak {g}, K)$-modules
 $X$ are discretely decomposable
 as $(\frak {g}^{\sigma}, K^{\sigma})$-modules?
\end{probalph}

Problem~\ref{prob:b} was solved previously
when $X$ is any Zuckerman derived functor 
module $A_{{\mathfrak {q}}}(\lambda)$,
 which is cohomologically induced from a $\theta$-stable
 parabolic subalgebra $\frak {q}$ of $\frak g_{\mathbb{C}}$:
 we gave a necessary and sufficient condition
 for discrete decomposability 
 in \cite{kob94,kob98ii}, 
and a complete classification
 of the triples $(\frak{g},\frak{g}^\sigma,\frak{q})$
 such that $A_\frak{q}(\lambda)$
 is discretely decomposable 
 as a $(\frak {g}^{\sigma}, K^{\sigma})$-module
 in a recent paper \cite{KoOs}.
We then observed that there exist symmetric pairs 
 $({\mathfrak {g}}, {\mathfrak {g}}^{\sigma})$
 for which none of $A_{\mathfrak {q}}(\lambda)$
 is discretely decomposable
 as a $(\frak {g}^{\sigma}, K^{\sigma})$-module
 except for $\frak {q}=\frak {g}_{\bb{C}}$.  
Even so, however, 
some other $(\frak g,K)$-modules
 $X$ might satisfy the property (1).  
This happens,
 for example,
 when $\frak g$ is a split real form 
 of $\frak {e}_6^{\bb C}$,
 $\frak {e}_7^{\bb C}$,
 and $\frak {e}_8^{\bb C}$.  
This observation has brought us to focus on Problem~\ref{prob:b}
 for some other \lq{small}\rq\ representations $X$
 as well.  
In particular,
we prove an easy-to-check criterion in Theorem~\ref{mindd}
 for discrete decomposability of a minimal representation $X$.  
These results
 serve as a part of the proof
 of our main results.  

One might wonder in Problem~\ref{prob:b}
 whether or not it is possible 
 to find such a unitarizable $(\frak g,K)$-module $X$
 if there exists at least one such
 (possibly, non-unitarizable) $X$.  
We shall show 
  in Corollary~\ref{unitary}
 that this is always possible.
(Notice
 that the classification 
 of unitarizable irreducible $(\frak g,K)$-modules
 is a long-standing problem
 in representation theory.  
 Fortunately,
 it turns out that previous achievements
 on this unsolved problem
 suffice to obtain Corollary~\ref{unitary}.)

Finally,
 we prove in Theorem~\ref{tensor} that the tensor product of
 two infinite-dimensional irreducible $(\frak{g},K)$-modules
 is discretely decomposable if and only if
 $G/K$ is a Hermitian symmetric space
 and these modules are simultaneously highest weight modules
 or they are simultaneously lowest weight modules.
This is in sharp contrast
 to the solution to Problem~\ref{prob:a}
 for symmetric pairs $(\frak {g}, \frak {g}^{\sigma})$
 with $\frak {g}$ simple:
 in this case there exist quite 
 often a family of irreducible $(\frak {g}, K)$-modules $X$
 that are discretely decomposable
 as $(\frak {g}^{\sigma}, K^{\sigma})$-modules
 but that are neither highest weight modules nor lowest weight modules.  

In unitary representation theory of real reductive groups,
it is in general a hard problem to find branching laws.
Even the tensor product of two irreducible unitary representations,
which is a special case of the restriction to symmetric pairs,
seems to be too wide for detailed analysis because of its wild feature
like infinite multiplicities in its irreducible decomposition for 
both discrete and continuous spectrum,
and the branching problem has been far from being solved.
Our motivation to highlight discretely decomposable restrictions is to single out 
a very nice framework among general branching problems for unitary representations,
in which we could expect a detailed study,
and in which even a purely algebraic approach
would make sense. 
See \cite{DV,GW,kob93,kob94,mfkorea,KO,KOP,S11}
 for explicit branching laws
 in various settings in the discretely decomposable case.   

The significance of discretely decomposable branching laws is not limited
inside representation theory in the narrow sense.
Restrictions of representations arise naturally from geometric settings.
In particular,
\textit{discretely decomposable} restrictions have found
their applications in the following different areas of mathematics in recent years:
\begin{itemize}
\item
modular varieties --- vanishing theorems of cohomologies \cite{KOd},
\item
construction of new discrete series for \textit{non-symmetric} spaces \cite{kob94},
\item
spectral analysis on \textit{non-Riemannian} locally symmetric spaces \cite{KaK},
\item
parabolic geometry --- construction of equivariant differential operators \cite{F}.
\end{itemize}

The $(\mathfrak{g},K)$-modules with the smallest Gelfand--Kirillov dimension
 (e.g.\ minimal representations) are especially interesting
 in the context of discretely decomposable restrictions.
As we will formulate precisely in this article, minimal representations are the most likely to
 be discretely decomposable.
Moreover, we can observe from some examples that the branching laws
 of minimal representations become fairly simple.
The Segal--Shale--Weil representation is a minimal representation
 of the metaplectic group and its restrictions to subgroups
 have been particularly well-studied in the connection of
 Howe's dual pair correspondence.  
The restrictions of the minimal representation of indefinite orthogonal
 group $O(p,q)$ to symmetric subgroups $O(p',q')\times O(p'',q'')$
 were studied in \cite{KO}.  
If one of $p'$, $q'$, $p''$ or $q''$ is zero,
 the minimal representation is discretely decomposable
 and is a direct sum of Zuckerman derived functor modules.
Moreover, the Parseval--Plancherel type theorem explains the unitary
structure for the minimal representations in terms of
those for Zuckerman derived functor modules.
The classification theorem of this article would suggest a further study
of both algebraic and analytic nature of branching laws 
for minimal representations in more general cases.

Here is a brief outline of the article.
Loosely speaking,
 the \lq{smaller}\rq\ $X$ is, 
the more likely the restriction $X|_{(\frak{h},H\cap K)}$
 is discretely decomposable.  
We formulate this feature by using associated varieties
 of $(\frak{g},K)$-modules.  
For this purpose,
 some basic properties
 of associated varieties
 are given in Section~\ref{sec:ass}.
We review a general necessary condition (Fact~\ref{nec})
 and a general sufficient condition (Fact~\ref{suff})
 for the discrete decomposability of restrictions 
 in Section~\ref{sec:dd}.
We apply them to the case that $X$ attains the minimum
 of the Gelfand--Kirillov
 dimension,
 and obtain a simple criterion
 for discrete decomposability
 in this case (Theorem~\ref{geldd}).
The main theorem (classification) is given in Section~\ref{sec:class}.
Concerning the tensor product of two irreducible representations,
 Problems~\ref{prob:a} and \ref{prob:b} are solved completely
 in Section~\ref{sec:tensor}.

\vskip 1pc
{\bf{$\langle$ Acknowledgement $\rangle$}}
\enspace
The authors would like to thank the American Institute of Mathematics
for supporting the workshop
 \lq\lq{Branching Problems for Unitary Representations\rq\rq}\
(July, 2011) and the Max Planck Institute for Mathematics
for their hospitality, where a part of this project was carried out.

\section{Preliminaries}
Let $G$ be a connected real reductive Lie group
 with Lie algebra ${\mathfrak {g}}$.  
We fix a Cartan decomposition 
${\mathfrak {g}} = {\mathfrak {k}} + {\mathfrak {p}}$, 
write 
$
     {\mathfrak {g}}_{\mathbb{C}}
 = {\mathfrak {k}}_{\mathbb{C}} + {\mathfrak {p}}_{\mathbb{C}}
$
 for its complexification,
 $
     {\mathfrak {g}}_{\mathbb{C}}^{\ast}
 = {\mathfrak {k}}_{\mathbb{C}}^{\ast} 
   + {\mathfrak {p}}_{\mathbb{C}}^{\ast}
$
 for the dual space, 
 and $K$ for the connected subgroup
 of $G$ with Lie algebra ${\mathfrak {k}}$.
We denote by $K_\bb{C}$ the subgroup of the inner automorphism group
 $\op{Int} \frak{g}_\bb{C}$ of $\frak{g}_\bb{C}$
 generated by $\op{exp}(\op{ad} (\frak{k}_\bb{C}))$.
Notice $K$ is not necessarily 
 a subgroup of $K_\bb{C}$, but there is a natural morphism
 $K\to K_\bb{C}$.  
The adjoint group $K_\bb{C}$ acts canonically
 on $\frak{p}_\bb{C}$
 and on the dual space $\frak{p}_\bb{C}^*$.
We take a Cartan subalgebra $\frak{t}$ of $\frak{k}$ and 
 choose a positive system $\Delta^+(\frak{k}_\bb{C},\frak{t}_\bb{C})$.
Let $B_K$ be the Borel subgroup of $K_\bb{C}$ corresponding
 to the positive roots.

Let ${\cal N}(\frak{g}^*_\bb{C})$ be the
 nilpotent variety of the dual space $\frak{g}^*_\bb{C}$, 
 and set
 ${\cal N}(\frak{p}^*_\bb{C})
:={\cal N}(\frak{g}^*_\bb{C})\cap \frak{p}^*_\bb{C}$.
By Kostant--Rallis \cite{KR}, 
 there are only finitely many $K_\bb{C}$-orbits
 in ${\cal N}(\frak{p}_\bb{C}^*)$ and each orbit is stable under
 multiplication by $\bb{C}^{\times}$.
Write the orbit decomposition as
 ${\cal N}({\frak{p}}_\bb{C}^*)
=\{0\}\sqcup \bb{O}_1\sqcup \cdots \sqcup \bb{O}_n$.
We say that $\bb{O}_i$ is {\it minimal} if it is minimal among 
$\bb{O}_1, \dots, \bb{O}_n$ with respect to the closure relation, 
 or equivalently, if the closure of $\bb{O}_i$
 is $\bb{O}_i\sqcup\{0\}$.

A simple Lie group $G$ 
 or its Lie algebra $\frak{g}$ is said to be of {\it Hermitian type}
 if the associated Riemannian symmetric space $G/K$
 is a Hermitian symmetric space,
 or equivalently, if the center $\frak{z}_K$ of $\frak{k}$ is one-dimensional.
If $G$ is of Hermitian type, the $K_\bb{C}$-module
 $\frak{p}_\bb{C}$ decomposes into two irreducible
 $K_\bb{C}$-modules:
 $\frak{p}_\bb{C}=\frak{p}_+ + \frak{p}_-$, 
 and $\frak{p}_-$ can be identified
 with the holomorphic tangent space at the base point in $G/K$.
The decomposition $\frak{p}^*_\bb{C}=\frak{p}^*_+ + \frak{p}^*_-$ 
 for the dual space is again an irreducible decomposition 
 as $K_\bb{C}$-modules.
The following notation will be used throughout the paper.

\begin{de}[highest non-compact root]
{\rm 
Let $G$ be a non-compact connected simple Lie group.
\label{beta}
\begin{enumerate}
\item[(1)]
If $G$ is not of Hermitian type, then $K_\bb{C}$ acts irreducibly
 on $\frak{p}^*_\bb{C}$ and we denote by 
 $\beta\in\sqrt{-1}\frak{t}^*$ the highest weight
 in $\frak{p}_\bb{C}^*$.
\item[(2)]
If $G$ is of Hermitian type, 
we label the $K_{\mathbb{C}}$-irreducible decomposition 
as $\frak{p}_\bb{C}=\frak{p}_+ + \frak{p}_-$
 and denote by $\beta\in\sqrt{-1}\frak{t}^*$ the highest weight
 in $\frak{p}^*_+$.
\end{enumerate}
}
\end{de}
Since $\frak{p}_\bb{C}^*$ is isomorphic to $\frak{p}_\bb{C}$
 as a $K_{\mathbb{C}}$-module
 by the Killing form,
 $-\beta$ is also a weight in $\frak{p}_\bb{C}^*$.  
For $\frak{g}$ of Hermitian type,
 $\frak{p}_-^*$ is dual to $\frak{p}_+^*$
 and therefore $-\beta$ occurs in $\frak{p}_-^*$.  
In either case, 
 the weight spaces $\frak{p}^*_\beta$ and $\frak{p}^*_{-\beta}$
 in $\frak{p}_\bb{C}^*$
 are one-dimensional.
Notice that $\beta$ depends
 on the labeling $\frak{p}_{\pm}$
 in Definition~\ref{beta} (2).  

Here is a description of minimal $K_\bb{C}$-orbits 
 in ${\mathcal{N}}(\frak{p}_\bb{C}^*)$.  
\begin{prop}
\label{minorbit}
Let $G$ be a non-compact connected simple Lie group.
\begin{enumerate}
\item[{\rm (1)}] 
If $G$ is not of Hermitian type, 
then there is a unique minimal $K_\bb{C}$-orbit in
 ${\cal N}(\frak{p}^*_\bb{C})$, which is given by
 $K_\bb{C}\cdot(\frak{p}^*_{\beta}\setminus\{0\})$.
\item[{\rm (2)}] 
If $G$ is of Hermitian type, 
then there are two minimal $K_\bb{C}$-orbits in
 ${\cal N}(\frak{p}^*_\bb{C})$, which are given by
 $K_\bb{C}\cdot (\frak{p}^*_{\beta}\setminus\{0\})$ and 
 $K_\bb{C}\cdot (\frak{p}^*_{-\beta}\setminus\{0\})$.
They have the same dimension.
\end{enumerate}
\end{prop}

\begin{proof}
(1)
Suppose that $G$ is not of Hermitian type.
Then $\frak{p}_\bb{C}^*$ is an irreducible $K_\bb{C}$-module
 with highest weight $\beta$.
Let $Z$ be a non-zero $K_\bb{C}$-stable subset of
 ${\cal N}(\frak{p}_\bb{C}^*)$.
To prove (1), it is enough to show that the closure $\overline{Z}$
 of $Z$ contains $\frak{p}^*_{\beta}$.
Take a non-zero element $x\in Z$ and
 write $x$ as the sum of $\frak{t}$-weight vectors in $\frak{p}^*_\bb{C}$:
 $x=\sum_{\alpha\in\Delta(\frak{p}^*_\bb{C},\frak{t}_\bb{C})} x_\alpha$.
Since any non-zero $B_K$-submodule of $\frak{p}^*_\bb{C}$
 contains $\frak{p}^*_\beta$, 
 we may assume that $x_\beta\neq 0$ by replacing
 $x$ by $bx$ $(b\in B_K)$.
We take an element $a\in\sqrt{-1}{\frak{t}}$ that is regular dominant
 with respect to
 $\Delta^+(\frak{k}_\bb{C},\frak{t}_\bb{C})$.
Then $\alpha(a)\in\bb{R}$ for any 
 $\alpha\in\Delta(\frak{p}^*_\bb{C},\frak{t}_\bb{C})$ and
 $\beta(a)>\alpha(a)$ if
 $\alpha\in\Delta(\frak{p}^*_\bb{C},\frak{t}_\bb{C})\setminus\{\beta\}$.
Define $x(s):=\exp(\op{ad}(sa))(x)\in\frak{p}^*_\bb{C}$
 for $s\in\bb{R}$.
Then 
\begin{align*}
{e^{-s\beta(a)}}x(s)
= x_\beta+
 \sum_{\alpha\in\Delta(\frak{p}^*_\bb{C},\frak{t}_\bb{C})\setminus \{\beta\}}
 {e^{s\alpha(a)-s\beta(a)}}x_\alpha
\end{align*}
and hence 
\[\lim_{s\to \infty} e^{-s\beta(a)}x(s)= x_\beta.\]
Since any $K_\bb{C}$-stable subset of
 ${\cal N}(\frak{p}^*_\bb{C})$ is
 stable under the multiplication by $\bb{C}^{\times}$,
 the vector $e^{-s\beta(a)}x(s)$ is contained in $Z$ for all
 $s\in\bb{R}$.
As a consequence, $\overline{Z}\ni x_\beta$ 
 and therefore $\overline{Z}\supset\frak{p}^*_\beta$.

(2)
Suppose that $G$ is of Hermitian type.
Then $\frak{p}^*_+$ is an irreducible $K_\bb{C}$-module
 with highest weight $\beta$.
Since $\frak{p}^*_-$ is its contragredient representation, 
 it is an irreducible $K_\bb{C}$-module 
 with lowest weight $-\beta$.
Let $Z$ be a non-zero $K_\bb{C}$-stable subset of
 ${\cal N}(\frak{p}_\bb{C}^*)$.
It is enough to show that
 $\overline{Z}\supset \frak{p}_{\beta}^*$ or
 $\overline{Z}\supset \frak{p}_{-\beta}^*$.
Take a non-zero element $x=x_++x_-\in Z$
 where $x_+\in\frak{p}^*_+$ and $x_-\in\frak{p}^*_-$.
We assume that $x_+\neq 0$.
Let $\frak{z}_K$ be the center of $\frak{k}$ and
 take an element $z\in\sqrt{-1}\frak{z}_K$
 such that
 $\op{ad}(z)=1$ on $\frak{p}^*_+$ and $\op{ad}(z)=-1$ on $\frak{p}^*_-$.
Then by an argument similar to the case (1) we have
\[\lim_{s\to \infty} {e^{-s}}\exp(\op{ad}(sz))x
 = x_+\]
and therefore $\overline{Z}\ni x_+$.
By using a similar argument again, we see that the closure of
 $K_\bb{C}\cdot x$ contains $x_\beta$ and hence
 $\overline{Z}\supset\frak{p}^*_\beta$.
If $x_+=0$, then $x_-\neq 0$ and we can prove
 similarly that $\overline{Z}\supset\frak{p}^*_{-\beta}$.

We can switch the two orbits
 $K_\bb{C}\cdot (\frak{p}^*_{\beta}\setminus\{0\})$ and 
 $K_\bb{C}\cdot (\frak{p}^*_{-\beta}\setminus\{0\})$
 by taking the complex conjugates with respect to the real form
 $\frak{g}$.
In particular they have the same dimension. 
\end{proof}

Proposition~\ref{minorbit} justifies
 the following notation:
\begin{alignat*}{2}
\bb{O}_{{\rm min}}:=&K_\bb{C}\cdot(\frak{p}^*_{\beta}\setminus\{0\}) \quad
&&\text{($\frak{g}$: not of Hermitian type),}
\\
\bb{O}_{{\rm min}, \pm}:=&K_\bb{C}\cdot(\frak{p}^*_{\pm\beta}\setminus\{0\}) 
\qquad&&\text{($\frak{g}$: Hermitian type)}.
\end{alignat*}
Their closures in $\frak{p}_\bb{C}^*$ are given by
\begin{align*}
\overline{\bb{O}_{\rm min}}= K_\bb{C}\cdot \frak{p}_{\beta}^*,\qquad
\overline{\bb{O}_{{\rm min}, \pm}}=K_\bb{C}\cdot\frak{p}_{\pm\beta}^*.
\end{align*}

They are related to the minimal nilpotent coadjoint orbit in the following way.
Suppose that $\frak{g}_\bb{C}$ is a complex simple Lie algebra.
This is equivalent to assuming that $\frak{g}$ is a simple real Lie algebra
 without complex structure.
Then there exists a unique non-zero minimal nilpotent
 $(\op{Int} \frak{g}_\bb{C})$-orbit
 in $\frak{g}_\bb{C}^*$, which we denote by
 $\bb{O}_{{\rm min}, \bb{C}}$. 
\begin{lem}
\label{orbitcap}
In the setting above, exactly one of the following    cases occurs.
\begin{enumerate}
\item[$(1)$]
$\bb{O}_{{\rm min}, \bb{C}}\cap \frak{p}_\bb{C}^*=\emptyset$.
\item[$(2)$]
$\frak{g}$ is not of Hermitian type and
 $\bb{O}_{{\rm min}, \bb{C}}\cap \frak{p}_\bb{C}^*=\bb{O}_{\rm min}$.
\item[$(3)$]
$\frak{g}$ is of Hermitian type and
$\bb{O}_{{\rm min}, \bb{C}}\cap \frak{p}_\bb{C}^*
 =\bb{O}_{{\rm min}, +}\cup\bb{O}_{{\rm min}, -}$.
\end{enumerate}
\end{lem}
This follows from the fact \cite{KR}:
for any $(\op{Int} \frak{g}_\bb{C})$-orbit $\bb{O}_\bb{C}$
 in the nilpotent variety ${\mathcal{N}}(\frak{g}_\bb{C}^*)$, 
 the intersection $\bb{O}_\bb{C}\cap\frak{p}_\bb{C}^*$
 is either empty or 
 the union of a finite number of equi-dimensional $K_\bb{C}$-orbits
 $\bb{O}_1, \dots, \bb{O}_m$, 
 and the dimension of $\bb{O}_j$ ($1 \le j \le m$) is equal
 to half the dimension of $\bb{O}_\bb{C}$.
We note
\begin{alignat*}{2}
\bb{O}_{{\rm min}, \bb{C}}
&\ne (\operatorname{Int}\frak{g}_{\bb {C}})\cdot\bb{O}_{{\rm min}}
&&\text{in Case (1),}
\\
\bb{O}_{{\rm min}, \bb{C}}
&= (\operatorname{Int} \frak{g}_{\bb {C}})\cdot\bb{O}_{{\rm min}}
&&\text{in Case (2),}
\\
\bb{O}_{{\rm min}, \bb{C}}
&= (\operatorname{Int} \frak{g}_{\bb {C}})\cdot\bb{O}_{{\rm min, +}}
 = (\operatorname{Int} \frak{g}_{\bb {C}})\cdot\bb{O}_{{\rm min, -}}
\qquad
&&\text{in Case (3)}.
\end{alignat*}
In Corollary~\ref{complex}, 
we provide six equivalent conditions
 to Case (1) 
 including a classification of such $\frak {g}$.  
 (Notice that
 the pair $(\frak{g}_\bb{C}, \frak{g})$ in Lemma~\ref{orbitcap}
 corresponds to $(\frak{g}, \frak{g}^\sigma)$ in
 the notation there.)

We define 
\[m(\frak{g}):=
\begin{cases}
 \op{dim}_\bb{C} {\bb{O}_{\rm min}} &
 \text{($\frak{g}$: not of Hermitian type),}\\
 \op{dim}_\bb{C} \bb{O}_{{\rm min},\pm} &
 \text{($\frak{g}$: of Hermitian type).}
\end{cases}
\]

For the reader's convenience,
 we list explicit values of $m(\frak{g})$.
By the Kostant--Sekiguchi correspondence,
 $m(\frak{g})$ coincides with half
 the dimension
 of the (real) minimal nilpotent coadjoint orbit(s)
 in $\frak{g}^{\ast}$.  
In Case (1) this is given in \cite{Ok}
 as follows:

\medskip
\begin{tabular}{c|ccccc}
$\frak{g}$ & $\frak{su}^*(2n)$ & $\frak{so}(n-1,1)$ & $\frak{sp}(m,n)$
 &  $\frak{f}_{4(-20)}$ & $\frak{e}_{6(-26)}$ \\ \hline
$m(\frak{g})$ & $4n-4$ & $n-2$ & $2(m+n)-1$ & $11$ & $16$
\end{tabular}
\medskip

In Case (2) and Case (3), 
 $m(\frak{g})$ is determined
 only by the complexified Lie algebra $\frak {g}_{\bb{C}}$:
 for $\frak{g}$ without complex structure we have
 $m(\frak{g})=\frac{1}{2}\op{dim}_\bb{C} \bb{O}_{{\rm min}, \bb{C}}$.
The dimension of the (complex) minimal nilpotent orbit
  $\op{dim}_\bb{C} \bb{O}_{{\rm min}, \bb{C}}$ is well-known.
Thus we have:

\medskip
\begin{tabular}{c|ccccccccc}
$\frak{g}_\bb{C}$ & $A_n$ & $B_n (n\geq 2)$
 & $C_n$ & $D_n$ & $\frak{g}_2^\bb{C}$ 
 & $\frak{f}_4^\bb{C}$ & $\frak{e}_6^\bb{C}$ & $\frak{e}_7^\bb{C}$
 & $\frak{e}_8^\bb{C}$ \\ \hline
$m(\frak{g})$ & $n$ & $2n-2$ & $n$ & $2n-3$ & $3$ & $8$ & $11$ & $17$ & $29$
\end{tabular}
\medskip

If $\frak{g}$ is a complex Lie algebra, $m(\frak{g})$ is twice the number
 above (e.g. $m(\frak{e}_{8}^\bb{C})=58$).

\section{Associated Varieties of $\frak {g}$-modules}
\label{sec:ass}

The associated varieties ${\mathcal{V}}_{\frak{g}}(X)$
 are a coarse approximation of 
 ${\mathfrak {g}}$-modules $X$, 
 which we brought in \cite{kob98ii}
 into the study of discretely decomposable restrictions
 of Harish-Chandra modules.
In this paper,
 we further develop its idea.  
For this, we collect some important properties of
 associated varieties that will be used in the later sections.

Let $\{U_j(\frak{g}_\bb{C})\}_{j \in {\mathbb{N}}}$
 be the standard increasing filtration
 of the universal enveloping algebra
 $U(\frak{g}_\bb{C})$.
Suppose $X$ is a finitely generated $\frak{g}$-module $X$.
A filtration $\{X_i\}_{i\in \bb{N}}$ is called a {\it good filtration} if
 it satisfies the following conditions.
\begin{itemize}
\item $\bigcup_{i\in \bb{N}} X_i=X$. 
\item $X_i$ is finite-dimensional for any $i\in \bb{N}$.
\item $U_j(\frak{g}_\bb{C})X_i\subset X_{i+j}$ for any $i,j\in\bb{N}$.
\item There exists $n$ such that $U_j(\frak{g}_\bb{C})X_i=X_{i+j}$
 for any $i\geq n$ and $j\in \bb{N}$.
\end{itemize}

The graded algebra
 $\op{gr} U(\frak{g}_\bb{C}):=
 \bigoplus_{j\in\bb{N}}U_j(\frak{g}_\bb{C})/U_{j-1}(\frak{g}_\bb{C})$
is isomorphic to the symmetric algebra $S(\frak{g}_\bb{C})$
 by the Poincar\'{e}--Birkhoff--Witt theorem and we 
 regard the graded module $\op{gr}X:=\bigoplus_{i\in\bb{N}} X_i/X_{i-1}$
 as an $S(\frak{g}_\bb{C})$-module.
Define
\begin{align*}
\op{Ann}_{S(\frak{g}_\bb{C})}(\op{gr}X):=
\{f\in S(\frak{g}_\bb{C}) : fv=0 \text{\ for any\ } v\in\op{gr}X\},
\\
{\cal V}_{\frak{g}}(X):=
\{x\in \frak{g}_\bb{C}^* : f(x)=0 \text{\ for any\ }
 f\in\op{Ann}_{S(\frak{g}_\bb{C})}(\op{gr}X)\}.
\end{align*}
Then ${\cal V}_{\frak{g}}(X)$ does not depend on
 the choice of good filtration and is called
 the {\it associated variety} of $X$.

The following basic properties on the associated variety are well-known.
\begin{lem}
\label{lem:ass}
Let $X$ be a finitely generated $\frak{g}$-module.
\begin{enumerate}
\item[$(1)$]
If $X$ is of finite length, then
 ${\cal V}_{\frak{g}}(X)\subset{\cal N}(\frak{g}_\bb{C}^*)$.
\item[$(2)$]
 ${\cal V}_{\frak{g}}(X)=0$ if and only if $X$ is finite-dimensional.
\item[$(3)$]
Let $\frak{h}$ be a Lie subalgebra of $\frak{g}_\bb{C}$.
Then ${\cal V}_{\frak{g}}(X)\subset \frak{h}^\perp$
 if $\frak{h}$ acts locally finitely on $X$,
 where 
 $\frak{h}^\perp:=
 \{x \in {\mathfrak {g}}_{\mathbb{C}}^{\ast}:
 x|_{\frak{h}}=0\}$.
\end{enumerate}
\end{lem}

(1) and (3) imply that
if $X$ is a $(\frak{g},K)$-module of finite length,
 then ${\cal V}_\frak{g}(X)$ is
 a $K_\bb{C}$-stable closed subvariety of
 ${\cal N}(\frak{p}^*_\bb{C})$
 because $\frak{k}_\bb{C}^\perp=\frak{p}_{\bb{C}}^{\ast}$.

We may recall
 that there is another well-known variety
 in $\frak{g}^*_\bb{C}$
 attached to a $\frak{g}$-module $X$ by using the annihilator ideal of
 $X$ in $U(\frak{g}_\bb{C})$.
Define the two-sided ideal
\[
\op{Ann} X:=
\{f\in U(\frak{g}_\bb{C}) : fv=0 \text{\ for any\ } v\in X\}
\]
and view the quotient $U(\frak{g}_\bb{C})/\op{Ann} (X)$
 as a $\frak{g}$-module by the product from left.
Then its associated variety 
 ${\cal V}_{\frak{g}}(U(\frak{g}_\bb{C})/\op{Ann} X)$
 is an $(\op{Int} \frak{g}_\bb{C})$-stable closed subvariety of
 $\frak{g}_\bb{C}^*$.
If $X$ is irreducible, 
 it is known that
 there is a unique nilpotent $(\op{Int} \frak{g}_\bb{C})$-orbit $\bb{O}_\bb{C}$
 in $\frak{g}_\bb{C}^*$ such that
 $\overline{\bb{O}_\bb{C}}=
 {\cal V}_{\frak{g}}(U(\frak{g}_\bb{C})/\op{Ann} X)$.
It should be noted
 that ${\cal V}_{\frak{g}}(X)$ has more information
 about the original $(\frak{g},K)$-module $X$
 than ${\cal V}_{\frak{g}}(U(\frak{g})/\op{Ann} X)$,
 and we shall use ${\cal V}_{\frak{g}}(X)$
 for the study of branching problems.  
A relation between ${\cal V}_{\frak{g}}(X)$ and
 ${\cal V}_{\frak{g}}(U(\frak{g}_\bb{C})/\op{Ann} X)$
 for a $(\frak{g},K)$-module $X$ is summarized as follows:

\begin{fact}[{\cite[Theorem 8.4]{Vo91}}]
\label{fact:assann}
Let $X$ be an irreducible $(\frak{g},K)$-module.
Let $\bb{O}_\bb{C}$ be as above.  
Then we have:
\begin{enumerate}
\item[$(1)$]
 ${\cal V}_\frak{g}(X)\subset
 {\cal V}_\frak{g}(U(\frak{g}_\bb{C})/\op{Ann} X)\cap\frak{p}_\bb{C}^*$.
\item[$(2)$]
 $\bb{O}_\bb{C}\cap \frak{p}_\bb{C}^*$ is the union of
 a finite number of $K_\bb{C}$-orbits
 $\bb{O}_1, \dots, \bb{O}_m$, each of which has dimension equal
 to half the dimension of $\bb{O}_\bb{C}$.
\item[$(3)$]
 Some of $\bb{O}_i$ are contained in ${\cal V}_\frak{g}(X)$ and
 they are precisely the $K_\bb{C}$-orbits of maximal dimension
 in ${\cal V}_\frak{g}(X)$.
\end{enumerate}
\end{fact}

The {\it Gelfand--Kirillov dimension} of $X$,
 to be denoted by $\op{DIM}(X)$, 
 is defined to be the dimension of ${\cal V}_{\frak{g}}(X)$, 
 or equivalently, 
 half the dimension
 of ${\cal V}_{\frak{g}}(U(\frak {g}_{\bb{C}})/\operatorname{Ann}X)$.
It follows from Proposition~\ref{minorbit} that
 any infinite-dimensional $(\frak{g},K)$-module $X$
 satisfies $\op{DIM}(X)\geq m(\frak{g})$.
The equality holds
 if and only if
\begin{align*}
{\cal V}_\frak{g}(X) =
\begin{cases}
 \overline{\bb{O}_{\rm min}} &
 \text{($\frak{g}$: not of Hermitian type),}\\
 \overline{\bb{O}_{{\rm min}, +}},
 \overline{\bb{O}_{{\rm min}, -}},
 {\text \ or\ } 
 \overline{\bb{O}_{{\rm min}, +}}
 \cup  \overline{\bb{O}_{{\rm min}, -}}&
 \text{($\frak{g}$:  of Hermitian type).}
\end{cases}
\end{align*}

Since ${\cal V}_\frak{g}(X)$ is $K_\bb{C}$-stable,
 the space of regular functions ${\cal O}({\cal V}_\frak{g}(X))$
 on ${\cal V}_\frak{g}(X)$
 can be viewed as a $K_\bb{C}$-module and hence as
 a $K$-module through the natural morphism $K\to K_\bb{C}$.
The following proposition shows that the
 $K$-type of a $(\frak{g},K)$-module $X$ can be approximated by 
 that of ${\cal O}({\cal V}_\frak{g}(X))$.
We write $X\underset{K}\leq Y$ for $K$-modules $X$ and $Y$
 if $\dim\op{Hom}_K(\tau, X)\leq \dim\op{Hom}_K(\tau, Y)$
 for any irreducible $K$-module $\tau$.

\begin{prop}
\label{prop:assk}
Let $X$ be a finitely generated $(\frak{g},K)$-module.
Then there exist finite-dimensional $K$-modules
 $F$ and $F'$ such that
\[
X|_K\underset{K}\leq{\cal O}({\cal V}_\frak{g}(X))\otimes F,
\quad \text{and} \quad
{\cal O}({\cal V}_\frak{g}(X))\underset{K}\leq X|_K\otimes F'.
\]
\end{prop}
\begin{proof}
Take a finite-dimensional $K$-submodule $X_0$ of $X$
 such that $U(\frak{g}_\bb{C})X_0=X$.
We get a good filtration
 $\{X_i:=U_i(\frak{g}_\bb{C})X_0\}_{i \in {\mathbb{N}}}$ of $X$, where
 $U_i(\frak{g}_\bb{C})$ is the standard increasing filtration
 of $U(\frak{g}_\bb{C})$.
The graded module
 $\op{gr} X:=\bigoplus_{i\in\bb{N}} X_i/X_{i-1}$ is a finitely generated 
 $S(\frak{p}_\bb{C})$-module
 and is isomorphic to $X$ as a $K$-module.
Let $I:=\sqrt{\op{Ann}_{S(\frak{p}_\bb{C})}(\op{gr} X)}$ be
 the radical of
 the annihilator of $\op{gr} X$.
Then $I$ is $K_{\bb{C}}$-stable and there is an isomorphism
\[{\cal O}({\cal V}_\frak{g}(X))\simeq S(\frak{p}_\bb{C})/I\]
 of $S(\frak{p}_\bb{C})$-modules, which respects the actions of $K_\bb{C}$.

Put $X'_j:=I^j\cdot\op{gr} X$ for $j\geq 0$.
Then there exists $n$ such that $X'_n=0$.
If $n$ is the smallest such integer,
 we get a finite filtration:
\[
0=X'_n\subset X'_{n-1}\subset \cdots \subset X'_0=\op{gr} X
\]
 and each successive quotient $X'_{j-1}/X'_j$ is
 an $(S(\frak{p}_\bb{C})/I)$-module.
Since $X'_{j-1}/X'_j$ is finitely generated
 as an $(S(\frak{p}_\bb{C})/I)$-module, 
 we can take
 a finite-dimensional $K$-submodule $F_j$
 of $X'_{j-1}/X'_j$ such that the map
 $(S(\frak{p}_\bb{C})/I)\otimes F_j\to X'_{j-1}/X'_j$
 is surjective.
Then we have
\[
X'_{j-1}/X'_j
\underset{K}\leq
(S(\frak{p}_\bb{C})/I)\otimes F_j
\simeq {\cal O}({\cal V}_\frak{g}(X))\otimes F_j
\]
and hence 
\[
X|_K\simeq
\bigoplus_{j=1}^n X'_{j-1}/X'_j
\underset{K}\leq
\bigoplus_{j=1}^n(S(\frak{p}_\bb{C})/I)\otimes F_j
\simeq {\cal O}({\cal V}_\frak{g}(X))\otimes \bigoplus_{j=1}^n F_j.
\]
The first inequality in the lemma follows by setting
 $F=\bigoplus_{j=1}^n F_j$.

Let us prove the opposite estimate.
We write ${\cal V}_\frak{g}(X)=Z_1\cup\dots\cup Z_m$
 for the irreducible decomposition, 
 and $P_i$ for the 
 defining ideal of $Z_i$ in $S(\frak {p}_{\bb{C}})$.
For each $i$,
 $Z_i$ and $P_i$ are $K_{\bb{C}}$-stable because $K_{\bb{C}}$ is connected.
Since $P_1,\dots, P_m$ are minimal prime ideals containing
 $\op{Ann}_{S(\frak{p}_\bb{C})}(\op{gr} X)$,
 they are associated primes of
 the $S(\frak{p}_\bb{C})$-module $\op{gr} X$
 (see \cite[Theorem 3.1]{Eis}).
This means that there exists
 an element $v_i\in \op{gr} X$ such that 
 the kernel of the map $S(\frak{p}_\bb{C}) \to \op{gr} X$, 
 $f\mapsto fv_i$ equals $P_i$. 
Let $F_i$ be a finite-dimensional $K$-submodule of
 $\op{gr} X$ that contains $v_i$. 
Then we get a map
\[\varphi_i:S(\frak{p}_\bb{C}) \to \op{Hom}_\bb{C}(F_i, \op{gr} X), \quad
 f \mapsto (v\mapsto fv),
\] 
 which respects the actions of $K$.
Let $\op{e}_{v_i}$ be the evaluation map
\[\op{e}_{v_i}:\op{Hom}_\bb{C}(F_i, \op{gr} X)\to \op{gr} X, \quad
 \alpha\mapsto \alpha(v_i).
\]
Then $\op{Ker}(\varphi_i)\subset
 \op{Ker}(\op{e}_{v_i}\circ\varphi_i)=P_i$.
As a consequence, 
\begin{align*}
{\cal O}({\cal V}_{\frak{g}}(X))
&\underset{K}\leq
\bigoplus_{i=1}^m{\cal O}(Z_i)
\simeq
\bigoplus_{i=1}^m S(\frak{p}_\bb{C})/P_i\\
&\underset{K}\leq
\bigoplus_{i=1}^m S(\frak{p}_\bb{C})/\op{Ker} (\varphi_i)
\underset{K}\leq
\bigoplus_{i=1}^m \op{Hom}_\bb{C}(F_i, \op{gr} X).
\end{align*}
By combining these inequalities with the natural isomorphisms of
 $K$-modules
\begin{align*}
\bigoplus_{i=1}^m \op{Hom}_\bb{C}(F_i, \op{gr} X)
\simeq
\bigoplus_{i=1}^m \op{gr} X\otimes F_i^*
\simeq
X|_K \otimes \bigoplus_{i=1}^m  F_i^*,
\end{align*}
we obtain the second inequality in the lemma by setting 
 $F'=\bigoplus_{i=1}^m F_i^*$.
\end{proof}

An irreducible $\frak{g}$-module $X$ is called
 a highest weight module
 if there exists a Borel subalgebra $\frak{b}$ of $\frak{g}_\bb{C}$
 such that $X$ has a one-dimensional $\frak{b}$-stable subspace.
If a simple Lie group $G$ allows an infinite-dimensional irreducible $(\frak{g},K)$-module
 which is simultaneously a highest weight module,
 then the group $G$ must be of Hermitian type
 and the Borel subalgebra $\frak{b}$ is compatible with
 the decomposition $\frak{p}_\bb{C}=\frak{p}_+ + \frak{p}_-$, namely, 
 either $\frak{b}\supset \frak{p}_+$ or $\frak{b}\supset\frak{p}_-$
 holds.

\begin{de}
{\rm 
Suppose that $G$ is a simple Lie group of Hermitian type.
An irreducible $(\frak{g},K)$-module $X$ is called
 a {\it highest weight $(\frak{g},K)$-module}
 (resp. {\it lowest weight $(\frak{g},K)$-module})
 if $X$ has a non-zero vector annihilated by
 $\frak{p}_+$
 (resp. $\frak{p}_-$).
}
\end{de}

The highest weight $(\frak{g},K)$-modules and
 the lowest weight $(\frak{g},K)$-modules are characterized
 by their associated varieties:

\begin{lem}
\label{lem:hwas}
Suppose that $G$ is a connected simple Lie group of Hermitian type, 
 and $X$ is an irreducible $(\frak{g},K)$-module.
Then $X$ is a highest weight $(\frak{g},K)$-module
 if and only if ${\cal V}_{\frak{g}}(X)\subset \frak{p}^*_-$.  
Likewise,
 $X$ is a lowest weight $(\frak{g},K)$-module
 if and only if ${\cal V}_{\frak{g}}(X)\subset \frak{p}^*_+$.
\end{lem}

\begin{proof}
If $X$ is a highest weight $(\frak{g},K)$-module,
 Lemma~\ref{lem:ass} (3) gives
 ${\cal V}_{\frak{g}}(X)\subset
 \frak{p}_\bb{C}^*\cap (\frak{p}_+)^\perp = \frak{p}^*_-$.

Suppose that ${\cal V}_{\frak{g}}(X)\subset\frak{p}^*_-$.
Then Proposition~\ref{prop:assk} yields an estimate of
 the $K$-type of $X$:
\begin{align}
\label{eq:kineq}
X|_K
\underset{K}\leq
 {\cal O}({\cal V}_{\frak{g}}(X))\otimes F
\underset{K}\leq
 {\cal O}(\frak{p}^*_-)\otimes F,
\end{align}
where $F$ is a finite-dimensional $K$-module.
Let $\frak{z}_K$ be the center of $\frak{k}$ and
 choose $z\in\sqrt{-1}\frak{z}_K$ 
 such that
 $\op{ad}(z)=1$ on $\frak{p}_+$ and
 $\op{ad}(z)=-1$ on $\frak{p}_-$.
Since ${\cal O}(\frak{p}^*_-)$ is isomorphic 
 to the symmetric algebra $S(\frak{p}_-)$, 
 the eigenvalues of $z$ on
 ${\cal O}(\frak{p}^*_-)$ are all negative.
By (\ref{eq:kineq}), 
 the set of eigenvalues of $z$ on
 $X$ is bounded above.
Hence there exists a maximal eigenvalue of $z$ and
 then $\frak{p}_+$ annihilates the corresponding
 eigenspace, which implies that
 $X$ is a highest weight $(\frak{g},K)$-module.

The proof for lowest weight $(\frak{g},K)$-modules
 is similar.
\end{proof}

\section{Discrete Decomposability}
\label{sec:dd}

Let $G$ be a real reductive Lie group and
 $\sigma$ an involutive automorphism of $G$.  
Then $\sigma$ induces involutions
 of the Lie algebra
 $\frak {g}$,
 its complexification $\frak {g}_{\bb{C}}$,
 the inner automorphism group $\op{Int}\frak{g}_{\bb{C}}$,
 etc.,
 for which we use the same letter $\sigma$.  
The subgroup $G^\sigma:=\{g\in G: \sigma(g)=g\}$
 is a reductive Lie group
 with Lie algebra 
 $\frak {g}^{\sigma}=\{x \in \frak {g}:\sigma (x)=x\}$, 
 and the pair $(G, H)$
 is called a reductive symmetric pair
 if $H$ is an open subgroup of $G^{\sigma}$.
Since the discrete decomposability
 of the restriction (see Definition~\ref{dd} below)
 does not depend on (finitely many)
 connected components of the subgroup,
 we shall consider the case $H=G^{\sigma}$
 without loss of generality.  
We can and do take a Cartan involution $\theta$
 of $G$ that commutes with $\sigma$.  
Then $\theta|_{G^{\sigma}}$ is a Cartan involution of $G^{\sigma}$.  
We set $K=G^{\theta}$ and $K^{\sigma}=G^\sigma \cap K$.  

The notion of discrete decomposability of $\frak{g}$-modules was
 introduced in \cite{kob98ii}.
We apply it to the restriction with respect to symmetric pairs, from
 $(\frak{g},K)$-modules to $(\frak{g}^\sigma, K^\sigma)$-modules.

\begin{de}
\label{dd}
{\rm 
A $(\frak{g},K)$-module $X$
 is said to be {\it{discretely decomposable as
 a $(\frak{g}^\sigma,K^\sigma)$-module}}
 if there exists an increasing filtration $\{X_i\}_{i\in \bb{N}}$ of
 $(\frak{g}^\sigma,K^\sigma)$-modules such that 
\begin{itemize}
\item $\bigcup_{i\in \bb{N}} X_i=X$ and
\item $X_i$ is of finite length as a $(\frak{g}^\sigma,K^\sigma)$-module
 for any $i\in \bb{N}$.
\end{itemize}
Discrete decomposability is preserved
 by taking submodules,
 quotients,
 and the tensor product 
 with finite-dimensional representations.  
}
\end{de}

\begin{rem}[see {\cite[Lemma 1.3]{kob98ii}}]
{\rm 
Suppose that $X$ is a unitarizable $(\frak{g},K)$-module.
Then $X$ is discretely decomposable  as a $(\frak{g}^\sigma,K^\sigma)$-module
if and only if 
$X$ is isomorphic to 
an algebraic direct sum of irreducible $(\frak{g},K)$-modules.
}
\end{rem}

We will state a necessary and a sufficient condition
 for the discrete decomposability, which were established in
 \cite{kob98i}, \cite{kob98ii}.

We write
\[\op{pr}:\frak{g}_\bb{C}^* \to {\frak{g}^{\sigma}_\bb{C}}^* \]
 for the restriction map.

\begin{fact}[necessary condition {\cite[Corollary 3.5]{kob98ii}}]
\label{nec}
Let $X$ be a $(\frak{g},K)$-module of finite length and suppose that
 $X$ is discretely decomposable
 as a $(\frak{g}^\sigma,K^\sigma)$-module.
Then 
$\op{pr} ({\cal V}_\frak{g} (X))\subset
 {\cal N}({\frak{g}_\bb{C}^\sigma}^*)$,
 where ${\cal N}({\frak{g}_\bb{C}^\sigma}^*)$ is the nilpotent variety
 of ${\frak{g}_\bb{C}^\sigma}^*$.
\end{fact}

We take a $\sigma$-stable Cartan subalgebra
 $\frak{t}=\frak{t}^\sigma+\frak{t}^{-\sigma}$
 of $\frak{k}$
 such that $\frak{t}^{-\sigma}$ is a maximal abelian subalgebra
 of $\frak{k}^{-\sigma}$.
We say a positive system $\Delta^+(\frak{k}_\bb{C}, \frak{t}_\bb{C})$
 is {\it{$(-\sigma)$-compatible}} 
 if $\{\alpha|_{\frak {t}_{\bb{C}}^{-\sigma}}:
 \alpha \in \Delta^+(\frak{k}_\bb{C}, \frak{t}_\bb{C})\}
 \setminus\{0\}$
 is a positive system 
 of the restricted root system 
 $\Sigma(\frak{k}_\bb{C},\frak{t}_\bb{C}^{-\sigma})$.  
Write $B_K$ for the Borel subgroup of $K_\bb{C}$ corresponding to 
  $\Delta^+(\frak{k}_\bb{C},\frak{t}_\bb{C})$.
If $\Delta^+(\frak{k}_\bb{C},\frak{t}_\bb{C})$ 
 is $(-\sigma)$-compatible,
 then $(K_\bb{C}^\sigma \cdot B_K)/B_K$
 is an open dense subset of the flag variety
 $K_\bb{C}/B_K$.
In Section~\ref{sec:dd} and Section~\ref{sec:class}, we always take
 a $(-\sigma)$-compatible
 positive system $\Delta^+(\frak{k}_\bb{C}, \frak{t}_\bb{C})$.

The asymptotic $K$-support $\op{AS}_K(X)$
 of a $K$-module $X$
 is a closed cone 
 in $\sqrt{-1} \frak {t}^{\ast} \setminus \{0\}$,
 which is defined as the limit cone 
 of the highest weights
 of irreducible $K$-modules
 occurring in $X$.  
The asymptotic $K$-support is preserved
 by taking the tensor product 
 of $X$ with a finite-dimensional representation.  
% (see \cite{kob98i} for the definition).
An estimate of the singularity spectrum
 of a hyperfunction character of $X$
 yields a criterion
 of \lq{$K'$-admissibility}\rq\ of $X$
 for a subgroup $K'$ of $K$.  
See \cite[Theorem 2.8]{kob98i}.  
When it is applied to the restriction with respect to reductive 
symmetric pairs $({\mathfrak {g}}, {\mathfrak {g}}^{\sigma})$  
we have:

\begin{fact}[sufficient condition {\cite[Example 2.14]{kob98i}}]
\label{suff}
Let $X$ be a $(\frak{g},K)$-module of finite length and suppose that
$\op{AS}_K(X)\cap \sqrt{-1}(\frak{t}^{\sigma})^\perp=\emptyset$.
Then $X$ is discretely decomposable as a $(\frak{g}^\sigma,K^\sigma)$-module.
\end{fact}

\begin{rem}
{\rm
Let $\theta$ be a Cartan involution of $G$ such that
 $\theta\sigma=\sigma\theta$.
Then $\theta\sigma$ becomes another involution of $G$ and
 the symmetric pair $(\frak{g},\frak{g}^{\theta\sigma})$
 is called the {\it associated pair} of $(\frak{g},\frak{g}^\sigma)$.
We note that $K^{\sigma} = K^{\theta\sigma}$.  
We can prove that
 a $(\frak {g},K)$-module $X$ is discretely decomposable
 as a $(\frak {g}^{\sigma},K^{\sigma})$-module
 if and only if it is discretely decomposable
 as a $(\frak {g}^{\theta\sigma},K^{\theta\sigma})$-module,
 though we do not use this in the paper.  
}
\end{rem}

In the rest of this section, we suppose that $G$ is a non-compact
 connected simple Lie group.

\begin{lem}
\label{projorbit}
Let $G$ be a non-compact connected simple Lie group
 and let $\beta$ be the highest non-compact root given in
 Definition~\ref{beta}.
Then $\op{pr}(K_\bb{C}\cdot\frak{p}^*_{\beta})\subset 
{\cal N}({\frak{p}_\bb{C}^\sigma}^*)$
 if and only if $\sigma\beta\neq -\beta$.
\end{lem}

\begin{proof}
Suppose that $\sigma\beta=-\beta$.
Take a non-zero vector $x\in\frak{p}^*_{\beta}$. 
Then $\sigma(x)\in \frak{p}^*_{-\beta}$, 
 $\overline{x}\in \frak{p}^*_{-\beta}$,
 and $\overline{\sigma(x)}\in \frak{p}^*_{\beta}$.
Here, $\overline{x}$ denotes the complex conjugate of $x$ with respect to
 the real form $\frak{g}$ of $\frak{g}_{\bb{C}}$.
Replacing $x$ by $cx$ $(c\in\bb{C})$ if necessary, we may assume that
 $y:=x+\overline{\sigma(x)}$ is non-zero.
Since $y\in\frak{p}^*_{\beta}$ and
 $\sigma(y)\in\frak{p}^*_{\sigma\beta}=\frak{p}^*_{-\beta}$, 
 the projection $\op{pr}(y)=\frac{1}{2}(y+\sigma (y))$ is non-zero.
We have moreover
\[\op{pr}(y)
=\frac{1}{2}(y+\sigma(y))
=\frac{1}{2}(x+\overline{x}+\sigma(x)+\overline{\sigma(x)})\in\frak{p}^*,\]
which is a semisimple element.
Therefore, $\op{pr}(y)\not\in {\cal N}({\frak{p}_\bb{C}^\sigma}^*)$ and 
hence $\op{pr}(K_\bb{C}\cdot\frak{p}^*_{\beta})\not\subset
{\cal N}({\frak{p}_\bb{C}^\sigma}^*)$.

Conversely, suppose that $\sigma\beta\neq -\beta$.
We can choose a vector $a\in \sqrt{-1}\frak{t}$
 such that $\beta(a)>0$ and $\sigma\beta(a)>0$.
This implies that
 the subspace $\frak{p}^*_{\beta}+\frak{p}^*_{\sigma\beta}$
 of $\frak{p}_{\bb{C}}^{\ast}$
 is contained in the nilradical of some Borel subalgebra of $\frak{g}_\bb{C}$.
In particular, 
 $\frak{p}^*_{\beta}+\frak{p}^*_{\sigma\beta}
 \subset{\cal N}({\frak{p}_\bb{C}^*})$
 and hence
 $\op{pr}(x)=\frac{1}{2}(x+\sigma(x))\in{\cal N}({\frak{p}_\bb{C}^\sigma}^*)$
 for $x\in\frak{p}_\beta^*$.
We regard $\frak{p}^*_\beta$ as a one-dimensional $B_K$-module and
 let $K_\bb{C}\times_{B_K}\frak{p}^*_{\beta}$ be 
 the $K_\bb{C}$-equivariant line bundle on the
 flag variety $K_\bb{C}/{B_K}$ with typical fiber $\frak{p}^*_{\beta}$.
Let $\mu: K_\bb{C}\times_{B_K}\frak{p}^*_{\beta} \to \frak{p}_\bb{C}^*$ be
 the map given by $[(k,x)]\mapsto k(x)$.  
Then, we have
 $\op{Image} \mu=K_\bb{C}\cdot\frak{p}^*_{\beta}$.
Let us consider the composition of the maps
\[K_\bb{C}\times_{B_K}\frak{p}^*_{\beta}
 \xrightarrow{\mu} \frak{p}_\bb{C}^*
 \xrightarrow{\op{pr}} {\frak{p}_\bb{C}^\sigma}^*.\]
Since
 $\op{pr}(x)\in{\cal N}({\frak{p}_\bb{C}^\sigma}^*)$
 for $x\in\frak{p}_\beta^*$
 and the composition $\op{pr}\circ\mu$ is $K_\bb{C}^\sigma$-equivariant,
 we have
 $\op{pr}\circ\mu([(k, x)])=k\cdot \op{pr}(x)
 \in{\cal N}({\frak{p}_\bb{C}^\sigma}^*)$
 for $k\in K_\bb{C}^\sigma$
 and $x\in\frak{p}_{\beta}^*$.
On the other hand, since we have chosen
 $\Delta^+(\frak{k}_\bb{C},\frak{t}_\bb{C})$
 to be $(-\sigma)$-compatible, 
 $(K_\bb{C}^\sigma\cdot B_K)/B_K$ is dense in $K_\bb{C}/B_K$.
Hence the subset
 $\{[(k, x)]: k\in K_\bb{C}^\sigma, x\in\frak{p}^*_{\beta}\}$
 is dense in $K_\bb{C}\times_{B_K}\frak{p}^*_{\beta}$.
We therefore have
\[\op{pr}(K_\bb{C}\cdot\frak{p}_\beta^*)=
\op{pr}\circ\mu(K_\bb{C}\times_{B_K}\frak{p}^*_{\beta})
\subset {\cal N}({\frak{p}_\bb{C}^\sigma}^*)\]
 because
 ${\cal N}({\frak{p}_\bb{C}^\sigma}^*)$ is closed in
 ${\frak{p}^\sigma_\bb{C}}^*$.
\end{proof}

\begin{prop}
\label{ddnec}
Let $X$ be an infinite-dimensional irreducible $(\frak{g},K)$-module.
If $X$ is discretely decomposable as a $(\frak{g}^\sigma,K^\sigma)$-module,
 then $\sigma\beta\neq -\beta$.
Here $\beta$ is the highest non-compact root given in Definition~\ref{beta}.
\end{prop}

\begin{proof}
The associated variety ${\cal V}_\frak{g}(X)$ is a non-zero
 $K_\bb{C}$-stable
 closed subset of $\frak{p}_\bb{C}^*$. 
By Proposition~\ref{minorbit},
 it follows that
 ${\cal V}_\frak{g}(X)\supset \overline{\bb{O}_{\rm min}}$
 if $\frak{g}$ is not of Hermitian type, and that 
 ${\cal V}_\frak{g}(X)\supset\overline{\bb{O}_{{\rm min}, +}}$ or 
 ${\cal V}_\frak{g}(X)\supset\overline{\bb{O}_{{\rm min}, -}}$
 if $\frak{g}$ is of Hermitian type.
In either case, we have
 ${\cal V}_\frak{g}(X)\supset K_\bb{C}\cdot \frak{p}^*_\beta$
 or 
 ${\cal V}_\frak{g}(X)\supset K_\bb{C}\cdot \frak{p}^*_{-\beta}$.
Hence 
 $\op{pr}(K_\bb{C}\cdot \frak{p}^*_\beta)\subset
 {\cal N}({\frak{p}_\bb{C}^\sigma}^*)$
 or
 $\op{pr}(K_\bb{C}\cdot \frak{p}^*_{-\beta})\subset
 {\cal N}({\frak{p}_\bb{C}^\sigma}^*)$
 by Fact~\ref{nec}.
For the former case,
 the claim $\sigma\beta\neq -\beta$
 follows from Lemma~\ref{projorbit}.
For the latter case, the claim can be proved
 by using an argument similar to
 the proof of Lemma~\ref{projorbit}. 
\end{proof}

The following lemma relates the asymptotic $K$-support
 to the associated variety
 of a $(\frak{g}, K)$-module. 

\begin{lem}
\label{askvar}
Let $X$ be a $(\frak{g},K)$-module of finite length.
Let ${\cal O}({\cal V}_\frak{g}(X))$ be the coordinate ring of
 the associated variety ${\cal V}_\frak{g}(X)$, 
 which is endowed
 with a natural $K_{\bb{C}}$-module structure and hence
 with a $K$-module structure through the morphism $K\to K_\bb{C}$.
Then we have
\[
\op{AS}_K(X)=\op{AS}_K({\cal O}({\cal V}_\frak{g}(X))).
\]
\end{lem}

\begin{proof}
This is immediate from Proposition~\ref{prop:assk}.
\end{proof}

\begin{lem}
\label{ask}
Let $\beta$ be as in Definition~\ref{beta}.  
Suppose $X$ is an irreducible $(\frak{g},K)$-module
 whose associated variety 
 ${\cal V}_\frak{g}(X)$ is equal to $K_\bb{C}\cdot\frak{p}_{\beta}^*$.
Then
\[
\op{AS}_K(X)=\bb{R}_{>0}(-w_0\beta)\equiv\{-sw_0\beta:s>0\},
\] where $w_0$ is
 the longest element of the Weyl group for
 $\Delta(\frak{k}_\bb{C},\frak{t}_\bb{C})$.
\end{lem}

\begin{proof}
Put $Z:=K_\bb{C}\cdot\frak{p}^*_{\beta}$. 
Let $S(\frak{p}_\bb{C})$ be the symmetric algebra of $\frak{p}_\bb{C}$, 
 which is identified with the space of regular functions on $\frak{p}_\bb{C}^*$. 
Write $I\subset S(\frak{p}_\bb{C})$
 for the defining ideal of $Z$ and
 write ${\cal O}(Z)$ for the coordinate ring of $Z$
 so that ${\cal O}(Z)\simeq S(\frak{p}_\bb{C})/I$.
By Lemma~\ref{askvar}, 
 it is enough to prove that
\begin{align}
\op{AS}_K({\cal O}(Z))= \bb{R}_{>0}(-w_0\beta).
\label{askoy}
\end{align}
Let $\mu: K_\bb{C}\times_{B_K}\frak{p}^*_{\beta} \to \frak{p}_\bb{C}^*$ be
 the map as in the proof of Lemma~\ref{projorbit}.
Since $\mu$ maps onto $Z$, the pull-back map
 $\mu^* :{\cal O}(Z) \to {\cal O}(K_\bb{C}\times_{B_K}\frak{p}^*_{\beta})$
 is injective.
As a representation of $B_K$,
 the contragredient representation of $\frak{p}^*_{\beta}$
 is isomorphic to $\bb{C}_{-\beta}$, the character of $B_K$
 corresponding to $-\beta\in\frak{t}_\bb{C}^*$. 
Therefore the regular functions
 ${\cal O}(K_\bb{C}\times_{B_K}\frak{p}^*_{\beta})$
are identified with the regular sections of the vector bundle
 $K_\bb{C}\times_{B_K}S(\bb{C}_{-\beta})$ on $K_\bb{C}/B_K$, 
where $S(\bb{C}_{-\beta})$ is the symmetric tensor of
 $\bb{C}_{-\beta}$.
By the Borel--Weil theorem, the space of regular sections of 
 $K_\bb{C}\times_{B_K}S^n(\bb{C}_{-\beta})$ is irreducible
 as a $K$-module and has highest weight $-nw_0\beta$. 
Hence (\ref{askoy}) follows.
\end{proof}

\begin{thm}
\label{geldd}
Let $G$ be a non-compact connected simple Lie group
 and suppose that $X$ is an infinite-dimensional
 irreducible $(\frak{g},K)$-module having the smallest Gelfand--Kirillov
 dimension, namely $\op{DIM}(X)=m(\frak{g})$.
Then $X$ is discretely decomposable as a $(\frak{g}^\sigma,K^\sigma)$-module
 if and only if $\sigma\beta\neq -\beta$, where $\beta$ is
 the highest non-compact root given
 in Definition~\ref{beta}.
\end{thm}

\begin{proof}
The \lq{only if}\rq\ part follows from Proposition~\ref{ddnec}.  

Conversely,
 suppose that $\sigma\beta\neq -\beta$.
We then have $\sigma w_0\beta\neq -w_0\beta$, where $w_0$ is the longest
 element of the Weyl group for $\Delta(\frak{k}_\bb{C},\frak{t}_\bb{C})$.
Indeed, since $K_\bb{C}\cdot\frak{p}_\beta=K_\bb{C}\cdot\frak{p}_{w_0\beta}$,
 Lemma~\ref{projorbit} shows
 $\op{pr}(K_\bb{C}\cdot\frak{p}_{w_0\beta})
\subset {\cal N}({\frak{p}_\bb{C}^\sigma}^*)$.
Then by using an argument similar to Lemma~\ref{projorbit} 
 we can prove that
 $\sigma w_0\beta\neq -w_0\beta$.
We prove the \lq{if}\rq\ part of the theorem
 in the case where $\frak{g}$ is of Hermitian type and
 ${\cal V}_{\frak{g}}(X)=\overline{\bb{O}_{{\rm min}, +}}
 \cup\overline{\bb{O}_{{\rm min}, -}}$.
Then as in the proof of Lemma~\ref{ask}, 
 we see that
 $\op{AS}_K(X)=\bb{R}_{> 0}(-w_0\beta)\cup\bb{R}_{> 0}\beta$.
Therefore, $\sigma\beta\neq -\beta$ and
 $\sigma w_0\beta\neq -w_0\beta$ imply that
 $\op{AS}_K(X)\cap \sqrt{-1}(\frak{t}^{\sigma})^\perp=\emptyset$.
Hence the theorem in this case follows from Fact~\ref{suff}.
The proof is similar for other cases. 
\end{proof}

For most non-compact simple Lie groups $G$, 
 there exist $(\frak{g},K)$-modules satisfying
 the assumption of Theorem~\ref{geldd} 
 (by replacing $G$ with a covering group of $G$
 if necessary).
However, 
 if $G$ is $SO_0(p,q)$ ($p+q$ : odd, $p,q \ge 4$) 
 or its covering group, then no
 irreducible $(\frak{g},K)$-module $X$
 satisfies
 $\op{DIM}(X)=m(\frak{g})$
 (see \cite{Vo81}).

A typical example of $(\frak{g},K)$-modules $X$
 that satisfy the assumption of Theorem~\ref{geldd}
 is a minimal representation.

\begin{de}
\label{minrep}
{\rm
Suppose that $G$ is a simple Lie group
 without complex structure.  
This means
 that the complexified Lie algebra
 $\frak{g}_\bb{C}$ is still 
 a simple Lie algebra.
An irreducible $(\frak{g},K)$-module $X$
 is said to be a {\it minimal representation} of $G$
 if the annihilator of the $U(\frak{g}_\bb{C})$-module $X$
 is the Joseph ideal of $U(\frak{g}_\bb{C})$ (\cite{Jos}).
}
\end{de}

By the definition of the Joseph ideal, we have:

\begin{prop}
\label{mingel}
Let $G$ be a connected simple Lie group without complex structure.
Suppose that $X$ is a minimal representation of $G$.
Then
\begin{align*}
{\cal V}_\frak{g}(X) =
\begin{cases}
 \overline{\bb{O}_{\rm min}} &
 \text{if $\frak{g}$ is not of Hermitian type,}\\
 \overline{\bb{O}_{{\rm min}, +}},
 \overline{\bb{O}_{{\rm min}, -}},
 {\text \ or\ } 
 \overline{\bb{O}_{{\rm min}, +}}
 \cup  \overline{\bb{O}_{{\rm min}, -}}&
 \text{if $\frak{g}$ is of Hermitian type.}
\end{cases}
\end{align*}
\end{prop}

\begin{proof}
Let $J$ be the Joseph ideal of $U(\frak{g}_\bb{C})$,
 which implies that ${\cal V}_\frak{g}(U(\frak{g}_\bb{C})/J)
 = \overline{\bb{O}_{{\rm min}, \bb{C}}}$.
Here $\bb{O}_{{\rm min}, \bb{C}}$ is
 the minimal nilpotent
 $(\op{Int} \frak{g}_\bb{C})$-orbit
 in $\frak{g}_\bb{C}^*$.
Then the proposition follows from
 Lemma~\ref{orbitcap} and Fact~\ref{fact:assann}.
\end{proof}

\begin{rem}
{\rm
Actually, we can sharpen Proposition~\ref{mingel} slightly: 
if $G$ is a connected simple Lie group of Hermitian type
 and $X$ is a minimal representation of $G$, then
 ${\cal V}_\frak{g} (X)$ is either $\overline{\bb{O}_{{\rm min}, +}}$
 or $\overline{\bb{O}_{{\rm min}, -}}$.
This is deduced from the following fact \cite{Vo91}:
 if $\bb{O}$ is a $K_\bb{C}$-orbit in
 ${\cal N}(\frak{p}_\bb{C}^*)$ and if
 $\overline{\bb{O}}$ is an irreducible component of
 ${\cal V}_\frak{g}(X)$,
 then at least one of the following two conditions 
 holds:
\begin{itemize}
\item
 ${\cal V}_\frak{g}(X)=\overline{\bb{O}}$, 
\item
 $\overline{\bb{O}}\setminus \bb{O}$ has codimension one
 in $\overline{\bb{O}}$.
\end{itemize}
}
\end{rem}

As a special case of Theorem~\ref{geldd}, 
 we obtain a criterion 
 for discrete decomposability
 of the restriction of minimal representations.

\begin{cor}
\label{mindd}
Let $G$ be a connected simple Lie group
 without complex structure.
Suppose that $G$ has a minimal representation $X$.
Then $X$ is discretely decomposable
 as a $(\frak{g}^\sigma, K^\sigma)$-module
 if and only if $\sigma\beta\neq -\beta$.
Here $\beta$ is the highest non-compact root given in Definition~\ref{beta}.
\end{cor}

\begin{rem}
{\rm
The converse statement of Proposition~\ref{mingel}
 is not true in general.
\begin{enumerate}
\item[(1)]
Let $G=SL(n,{\mathbb{R}})$.
The Joseph ideal of $U(\frak{g}_\bb{C})$ is not defined
 for $\frak{g}_\bb{C}=\frak{sl}(n,\bb{C})$, 
 but
 there exists an irreducible $(\frak {g}, K)$-module $X$
 isomorphic to the underlying $(\frak{g},K)$-module
 of some degenerate principal series representation such that
 ${\cal V}_\frak{g}(X)=\overline{\bb{O}_{\rm min}}$.
\item[(2)]
Let $G=Sp(m,n)$. Then $\bb{O}_{\rm min, \bb{C}}$ does not
 intersect with $\frak{p}^*_\bb{C}$
 (see Corollary~\ref{complex}).
From Fact~\ref{fact:assann},
 there exists no minimal representation of $G$.
However, there exists an irreducible
 $(\frak{g},K)$-module $X$ isomorphic to
 some $A_\frak{q}(\lambda)$ such that
 ${\cal V}_\frak{g}(X)=\overline{\bb{O}_{\rm min}}$.
\item[(3)]
If $X$ is a minimal representation,
 then any infinite-dimensional $(\frak{g},K)$-module
 in its coherent family has the same associated variety
 as $X$.
However, most of them are not a minimal representation
 because a minimal representation must have a fixed
 infinitesimal character.
\end{enumerate}
Theorem~\ref{geldd} can be applied to these representations
 as well.
}
\end{rem}

%%%%%%%%%%%%%%%%%%%%%%%%%%%%%%%%%%%%%%%%%%%%%%%%%%
%%%%%%%%%%%%%%%%%%%%%%%%%%%%%%%%%%%%%%%%%%%%%%%%%%
%%%%%%%%%%%%%%%%%%%%%%%%%%%%%%%%%%%%%%%%%%%%%%%%%%

\section{Classification}
\label{sec:class}

In this section
 we assume $G$ to be a non-compact
 connected simple Lie group.  
Let $K$ be the connected subgroup of $G$
 associated to a Cartan decomposition $\frak{g}=\frak{k}+\frak{p}$.
The Cartan involution $\theta$
 is chosen to satisfy $\sigma \theta =\theta \sigma$
 and the positive system
 $\Delta^+(\frak{k}_\bb{C},\frak{t}_\bb{C})$
 is chosen to be $(-\sigma)$-compatible
 if an involutive automorphism $\sigma$
 of $G$ is given.  

\begin{de}
\label{de:holpair}
{\rm
Let $\frak{g}$ be a non-compact real simple Lie algebra
 and $(\frak{g}, \frak{g}^\sigma)$ a symmetric pair.
We say $(\frak{g}, \frak{g}^\sigma)$ is of {\it holomorphic type}
 if $\frak{g}$ is of Hermitian type and the center $\frak{z}_K$
 of $\frak{k}$ is contained in $\frak{g}^\sigma$, 
 or equivalently,
 $\sigma$ induces
 a holomorphic involution
 on the Hermitian symmetric space $G/K$.
}
\end{de}

For example, 
 the symmetric pairs $(\frak{sp}(n,\bb{R}), \frak{u}(m,n-m))$ and
 $(\frak{sp}(n,\bb{R}), \frak{sp}(m,\bb{R})\oplus\frak{sp}(n-m,\bb{R}))$
 are of holomorphic type for any $m$ and $n$,
 whereas the symmetric pair $(\frak{sp}(n,\bb{R}), \frak{gl}(n,\bb{R}))$
 is not of holomorphic type. 

Here is the main result of this paper:

\begin{thm}[classification]
\label{thm:ddpair}
Let $\frak{g}$ be a non-compact real simple Lie algebra
 and $(\frak{g}, \frak{g}^\sigma)$ a symmetric pair.
The following three conditions on the symmetric pair
 $(\frak{g}, \frak{g}^\sigma)$
 are equivalent:
\begin{enumerate}
\item[{\rm(i)}]
There exists an infinite-dimensional irreducible $(\frak{g},K)$-module $X$
 (by replacing $G$ with a covering group of $G$ if necessary) 
 such that $X$
 is discretely decomposable as a $(\frak{g}^\sigma,K^\sigma)$-module.
\item[{\rm(ii)}] $\sigma\beta\neq -\beta$
 ($\beta$ is the highest non-compact root given in Definition~\ref{beta}).
\item[{\rm(iii)}] The pair $(\frak{g}, \frak{g}^\sigma)$ satisfies one of the
 following.
\begin{enumerate}
\item[{\rm(a)}] $\sigma$ is a Cartan involution,
 i.e.\ $\frak{g}^\sigma=\frak{k}$.
\item[{\rm(b)}] $(\frak{g}, \frak{g}^\sigma)$ is of holomorphic type
 (see \cite[Table 2]{KoOs} for a classification
 of symmetric pairs of holomorphic type).
\item[{\rm(c)}] The pair $(\frak{g},\frak{g}^\sigma)$
 appears in Table~\ref{ddpairs}
 (up to isomorphisms).  
\end{enumerate}
\end{enumerate}
\end{thm}

%%%%%%%%%%%%%%%%%%%%%%%%%%%%%%%%%%%%%%%%%%%%%%%%%%%%%%

\begin{table}[!h]
\begin{center}
\begin{tabular}{cccc}
\hline
\rule[-5pt]{0pt}{7pt}
$\qquad\frak{g}\qquad$&$\qquad\frak{g}^\sigma\qquad$
&minimal&$A_\frak{q}(\lambda)$\\
\hline

\rule[0pt]{0pt}{12pt}
$\frak{sl}(2n,\mathbb{R})$&{
\rule[-5pt]{0pt}{5pt}
$\frak{sl}(n,\mathbb{C})\oplus \frak{u}(1),
 \quad \frak{sp}(n,\mathbb{R})$}
 &&$n=2$\\
\hline

\rule[0pt]{0pt}{12pt}
$\frak{su}(2m,2n)$&{
\rule[-5pt]{0pt}{5pt}
$\frak{sp}(m,n)$}&&$\bigcirc$\\
\hline

\rule[0pt]{0pt}{12pt}
$\frak{so}(m,n)$&{
\rule[-5pt]{0pt}{5pt}
$\frak{u}(\frac{m}{2},\frac{n}{2})$}&$(*)$&$\bigcirc$\\
&{\rule[-5pt]{0pt}{5pt}
$\frak{so}(m,k)\oplus\frak{so}(n-k)\ (m> 1)$}&$(*)$&$\bigcirc$\\ 
\hline

\rule[0pt]{0pt}{12pt}
$\frak{sp}(2n,\mathbb{R})$&{
\rule[-5pt]{0pt}{5pt}
$\frak{sp}(n,\mathbb{C})$}&$\bigcirc$&$n=1$\\ 
\hline

\rule[0pt]{0pt}{12pt}
$\frak{sp}(m,n)$&{
\rule[-5pt]{0pt}{5pt}
$\frak{sp}(k,l)\oplus\frak{sp}(m-k,n-l)$}&&$\bigcirc$\\ 
\hline

\rule[0pt]{0pt}{12pt}
$\frak{sl}(2n,\mathbb{C})$&{
\rule[-5pt]{0pt}{5pt}
$\frak{sp}(n,\mathbb{C}), \quad \frak{su}^*(2n)$}&&$\bigcirc$\\
\hline

\rule[0pt]{0pt}{12pt}
$\frak{so}(n,\mathbb{C})$&{
\rule[-5pt]{0pt}{5pt}
$\frak{so}(n-1,\mathbb{C}), \quad \frak{so}(n-1,1)$}\quad $(n\geq 5)$
&&$n$ : even\\ 
\hline

\rule[0pt]{0pt}{12pt}
$\frak{sp}(n,\mathbb{C})$&{
\rule[-5pt]{0pt}{5pt}
$\frak{sp}(k,\mathbb{C})\oplus\frak{sp}(n-k,\mathbb{C}),
\quad \frak{sp}(k,n-k)$}&&\\ 
\hline

\rule[0pt]{0pt}{12pt}
$\frak{f}_{4(4)}$&{
\rule[-5pt]{0pt}{5pt}
$\frak{sp}(2,1)\oplus \frak{su}(2), \quad \frak{so}(5,4)$}
&$\bigcirc$&$\bigcirc$\\ 
\hline

\rule[0pt]{0pt}{12pt}
$\frak{f}_{4(-20)}$&{
\rule[-5pt]{0pt}{5pt}
$\frak{so}(8,1)$}&&$\bigcirc$\\
\hline

\rule[0pt]{0pt}{12pt}
$\frak{e}_{6(6)}$&{
\rule[-5pt]{0pt}{5pt}
$\frak{su}^*(6)\oplus\frak{su}(2),\quad \frak{f}_{4(4)}$}
&$\bigcirc$&\\
\hline

\rule[0pt]{0pt}{12pt}
$\frak{e}_{6(2)}$&{
\rule[-5pt]{0pt}{5pt}\ \ 
$\frak{so}(6,4)\oplus \frak{so}(2),\quad \frak{su}(4,2)\oplus \frak{su}(2)$}
&$\bigcirc$&$\bigcirc$\\ 
&{\rule[-5pt]{0pt}{5pt}\ \ 
$\frak{sp}(3,1),\quad \frak{f}_{4(4)}$}
&$\bigcirc$&$\bigcirc$\\ 
&{\rule[-5pt]{0pt}{5pt}\ \ 
$\frak{so}^*(10)\oplus \frak{so}(2)$}
&$\bigcirc$&$\bigcirc$\\
\hline

\rule[0pt]{0pt}{12pt}
$\frak{e}_{6(-14)}$&{
\rule[-5pt]{0pt}{5pt}
$\frak{f}_{4(-20)}$}
&$\bigcirc$&$\bigcirc$\\
\hline

\rule[0pt]{0pt}{12pt}
$\frak{e}_{7(7)}$&{
\rule[-5pt]{0pt}{5pt}
$\frak{so}^*(12)\oplus\frak{su}(2),\quad \frak{e}_{6(2)}\oplus\frak{so}(2)$}
&$\bigcirc$&\\
\hline

\rule[0pt]{0pt}{12pt}
$\frak{e}_{7(-5)}$&{
\rule[-5pt]{0pt}{5pt}
$\frak{su}(6,2),\quad \frak{e}_{6(2)}\oplus\frak{so}(2)$}
&$\bigcirc$&$\bigcirc$\\
&{\rule[-5pt]{0pt}{5pt}
$\frak{so}(8,4)\oplus\frak{su}(2)$}
&$\bigcirc$&$\bigcirc$\\
&{\rule[-5pt]{0pt}{5pt}
$\frak{e}_{6(-14)}\oplus\frak{so}(2)$}
&$\bigcirc$&$\bigcirc$\\
\hline

\rule[0pt]{0pt}{12pt}
$\frak{e}_{8(8)}$&{
\rule[-5pt]{0pt}{5pt}
$\frak{e}_{7(-5)}\oplus\frak{su}(2)$}
&$\bigcirc$&\\
\hline

\rule[0pt]{0pt}{12pt}
$\frak{e}_{8(-24)}$&{
\rule[-5pt]{0pt}{5pt}
$\frak{so}(12,4),\quad \frak{e}_{7(-5)}\oplus\frak{su}(2)$}
&$\bigcirc$&$\bigcirc$\\
\hline

\rule[0pt]{0pt}{12pt}
$\frak{f}_{4}^{\mathbb{C}}$&{
\rule[-5pt]{0pt}{5pt}
$\frak{so}(9,\mathbb{C}),\quad \frak{f}_{4(-20)}$}&&\\
\hline

\rule[0pt]{0pt}{12pt}
$\frak{e}_{6}^{\mathbb{C}}$&{
\rule[-5pt]{0pt}{5pt}
$\frak{f}_{4}^{\mathbb{C}},\quad \frak{e}_{6(-26)}$}&&\\
\hline

\end{tabular}
\end{center}
\caption{}
\label{ddpairs}
\end{table}

\begin{rem}
\rm{
In Table~\ref{ddpairs},
 a symmetric pair and its associated pair
 are listed in the same row.  
For example, we list two symmetric pairs 
 $(\frak{sl}(2n,\bb{R}), \frak{sl}(n,\bb{C})\oplus\frak{u}(1))$,
 $(\frak{sl}(2n,\bb{R}), \frak{sp}(n,\bb{R}))$ in the first row
 and one is the associated pair of the other.
In the second row, only one symmetric pair
 $(\frak{su}(2m,2n),\frak{sp}(m,n))$ is listed.
This means that the pair $(\frak{su}(2m,2n),\frak{sp}(m,n))$
 is self-associated.
}
\end{rem}

\begin{rem}
\label{remtable}
{\rm
Here is a guidance to the notation 
 used in Table~\ref{ddpairs}.
\begin{enumerate}
\item[(1)]
The circle $\bigcirc$ below \lq\lq{minimal}\rq\rq\ means
 that there exists a minimal representation
 for some Lie group $G$
 with Lie algebra ${\mathfrak {g}}$.  
For these pairs in Table~\ref{ddpairs},
 a minimal representation $X$ is discretely decomposable
 as a $(\frak {g}^{\sigma},K^{\sigma})$-module
 by Corollary~\ref{mindd}, 
 and thus the condition (i) is fulfilled.  

The asterisk ($\ast$) for $\frak {g}=\frak {so}(m,n)$ reflects
 the fact
 that the existence of minimal representations
 depends on the parameters $m$ and $n$:
 there exists a minimal representation
 for some Lie group $G$
 with Lie algebra $\frak{so}(m,n)$
 if and only if $(m,n)$ satisfies
 one of the following.
\begin{itemize}
\item
$m+n$ is even, $m,n\geq 2$, and $m+n\geq 8$.
\item
$(m,n)=(3,2l), (2l,3)$ for $l\geq 2$.
\item
$(m,n)=(2,2l+1), (2l+1,2)$ for $l\geq 1$.
\end{itemize}
\item[(2)]
The circle $\bigcirc$ below
 ``$A_\frak{q}(\lambda)$" means
 that there exists a $\theta$-stable parabolic subalgebra
 $\frak{q}$ ($\ne \frak{g}_\bb{C}$)
 such that the Zuckerman derived functor modules 
 $A_\frak{q}(\lambda)$ 
 are discretely decomposable
 as $(\frak{g}^{\sigma},K^{\sigma})$-modules.
\item[(3)]
For real exceptional Lie algebras,
 we follow the notation of \cite[Chapter X]{hel}.
\end{enumerate}
}
\end{rem}

\begin{rem}
{\rm
We did not intend to make the conditions (a), (b), and (c)
 in Theorem~\ref{thm:ddpair} to be exclusive with one another.
For example, the pair
 $(\frak{so}(m, n), \frak{u}(\frac{m}{2},\frac{n}{2}))$ is
 of holomorphic type if $m=2$.
}
\end{rem}

Before giving a proof of Theorem~\ref{thm:ddpair},
 we prepare the following:
\begin{lem}
\label{ddnothol}
Let $\frak{g}$ be a non-compact real simple Lie algebra.  
Assume that the symmetric pair $(\frak{g}, \frak{g}^\sigma)$
 is not of holomorphic type.  
Then the following three conditions
 on  $\sigma$ are equivalent:
\begin{enumerate}
\item[{\rm(i)}] $\sigma\beta\neq -\beta$.
\item[{\rm(ii)}]
$-\sigma \beta$ is not dominant 
 with respect to $\Delta^+(\frak{k}_{\bb C}, \frak{t}_{\bb C})$.  
\item[{\rm(iii)}] 
The pair $(\frak{g}, \frak{g}^\sigma)$ satisfies (a) or (b).
\begin{enumerate}
\item[{\rm(a)}] $\sigma$ is a Cartan involution,
  i.e.\ $\frak{g}^\sigma=\frak{k}$.
\item[{\rm(b)}]  The pair $(\frak{g},\frak{g}^\sigma)$
 appears in Table~\ref{ddpairs}
 (up to isomorphisms).  
\end{enumerate}
\end{enumerate}
\end{lem}

\begin{proof}
(i) $\Leftrightarrow$ (ii) \enspace
We set 
\[
  V:=\begin{cases}
     \frak{p}_\bb{C}^* \qquad
     &\text{if $\frak{g}$ is not of Hermitian type}, 
     \\
     \frak{p}_+^*       
     &\text{if $\frak{g}$ is of Hermitian type.  }
     \end{cases}
\]
Then $K_{\bb{C}}$ acts irreducibly on $V$ 
 and $\beta$ is the highest weight of $V$.  
We claim that the set $\Delta(V, \frak{t}_\bb{C})$
 of weights is preserved by $-\sigma$.  
In fact, 
 if $\frak{g}$ is not of Hermitian type,
 $\sigma \frak{p}_{\bb{C}}^*= \frak{p}_{\bb{C}}^*$ 
and hence
$
  \sigma(\Delta(\frak{p}_{\bb{C}}^*, \frak{t}_{\bb{C}}))
    = \Delta(\frak{p}_{\bb{C}}^*, \frak{t}_{\bb{C}})
    = - \Delta(\frak{p}_{\bb{C}}^*, \frak{t}_{\bb{C}}).  
$
If $\frak{g}$ is of Hermitian type,
 let $\frak{z}_K$ be the center of $\frak{k}$.
Then $\sigma(z)=-z$ for $z\in\sqrt{-1}\frak{z}_K$
 because $(\frak{g},\frak{g}^\sigma)$
 is not of holomorphic type.
Hence $\sigma \frak{p}_+^*=\frak{p}_-^*$
 and 
$\sigma(\Delta(\frak{p}_+^*, \frak{t}_{\bb{C}}))
 =\Delta(\frak{p}_-^*, \frak{t}_{\bb{C}})
 =-\Delta(\frak{p}_+^*, \frak{t}_{\bb{C}})$.  
Thus $-\sigma(\Delta(V, \frak{t}_{\bb{C}}))
 =\Delta(V, \frak{t}_{\bb{C}})$
 in either case. 

Since $\beta$, $-\sigma \beta \in \Delta(V, \frak{t}_{\bb{C}})$
 are of the same length, 
 $-\sigma \beta$ is dominant 
 if and only if $-\sigma\beta$ coincides
 with the highest weight $\beta$
 of the irreducible representation $V$.  
Hence the equivalence (i) $\Leftrightarrow$ (ii) is proved.  

(ii) $\Leftrightarrow$ (iii)\enspace
We recall that a classification
 of symmetric pairs
 with $-\sigma\beta$ dominant 
 was carried out in \cite{KoOs}.
In the case that $(\frak{g},\frak{g}^\sigma)$ is
 a symmetric pair not 
 of holomorphic type, 
 the weight $-\sigma\beta$ is dominant
 if and only if
 the real form $\frak{k}^\sigma+\sqrt{-1}\frak{k}^{-\sigma}$
 is split 
 or $(\frak{g},\frak{g}^\sigma)$ is one of those listed
 in \cite[Appendix B.1]{KoOs}.
Consequently, $(\frak{g},\frak{g}^\sigma)$
 satisfies $-\sigma\beta\neq\beta$ if and only if the following two
 conditions hold:
\begin{itemize}
\item
 $\frak{k}^\sigma+\sqrt{-1}\frak{k}^{-\sigma}$
 is not a split real form of $\frak{k}_\bb{C}$;
\item
 $(\frak{g},\frak{g}^\sigma)$ is not listed in \cite[Appendix B.1]{KoOs}.
\end{itemize}
Table~\ref{ddpairs} is obtained 
 as the complementary subset
 of these pairs
 in all the symmetric pairs with $\frak{g}$ simple
 (not of holomorphic type), 
 for which the classification
 was established
 earlier by M.~Berger \cite{ber}.  
Hence
 the equivalence (ii) $\Leftrightarrow$ (iii) is proved.
\end{proof}

We are ready to prove Theorem~\ref{thm:ddpair}.  

\begin{proof}[Proof of Theorem~\ref{thm:ddpair}]

(i) $\Rightarrow$ (ii) \enspace
This is Proposition~\ref{ddnec}.

(ii) $\Leftrightarrow$ (iii) \enspace
If the pair $(\frak{g},\frak{g}^\sigma)$ is of holomorphic type, 
 then we can take a non-zero element
 $z$ in the center $\frak{z}_K$ of $\frak{k}$ and 
 we have $\beta(z)\neq 0$.
Since $\sigma$ acts as the identity on $\frak{z}_K$,
 it follows that
 $(\sigma \beta)(z)=\beta(\sigma (z))=\beta(z)$
 and hence $-\sigma\beta\neq \beta$.

If the pair $(\frak{g},\frak{g}^\sigma)$ is not of holomorphic type,
 our assertion follows from Lemma~\ref{ddnothol}.

(iii) $\Rightarrow$ (i) \enspace
To prove this implication,
 we have to find a discretely decomposable
 $(\frak{g},K)$-module $X$.
If $\frak{g}^\sigma=\frak{k}$,
 then any irreducible $(\frak{g},K)$-module
 is discretely decomposable as a $(\frak{g}^\sigma, K^\sigma)$-module.
If $(\frak{g},\frak{g}^\sigma)$ is of holomorphic type,
 then $\frak{g}$ is of Hermitian type and there exist
 infinite-dimensional highest weight $(\frak{g},K)$-modules.
It is known that 
 any highest weight $(\frak{g},K)$-module
 is discretely decomposable as a $(\frak{g}^\sigma, K^\sigma)$-module
 if $(\frak{g},\frak{g}^\sigma)$ is of holomorphic type
 (see \cite[Theorem 7.4]{kob98iii}).

Suppose that the pair $(\frak{g},\frak{g}^\sigma)$
 is isomorphic to one of those listed in Table~\ref{ddpairs}.
We give three sufficient conditions for (i):
\begin{enumerate}
\item[(1)]
There exists a minimal representation $X$
 for some connected covering group of $G$.
\item[(2)]
$\frak{g}$ is a complex Lie algebra
 and $\frak{g}\not\simeq\frak{sl}(n,\bb{C})$.
\item[(3)]
There exists a $\theta$-stable parabolic subalgbra
 $\frak{q} (\neq \frak{g}_\bb{C})$
 such that
 the Zuckerman derived functor modules $A_\frak{q}(\lambda)$
 are discretely decomposable
 as $(\frak{g}^\sigma, K^\sigma)$-modules.
\end{enumerate}

(1) is satisfied
 for $\frak{g}=\frak{so}(m,n)$ with a certain condition on $m,n$
 (see Remark~\ref{remtable}), 
 $\frak{g}=\frak{sp}(2n,\bb{R})$, $\frak{f}_{4(4)}$, $\frak{e}_{6(6)}$,
 $\frak{e}_{6(2)}$, $\frak{e}_{6(-14)}$,
 $\frak{e}_{7(7)}$, $\frak{e}_{7(-5)}$, $\frak{e}_{8(8)}$,
 $\frak{e}_{8(-24)}$ (see \cite{To}).
Then by Theorem~\ref{mindd}, a minimal representation $X$ 
 is discretely decomposable as a $(\frak{g}^\sigma,K^\sigma)$-module.

If (2) holds, then
 put $X=U(\frak{g})/J$, where $J$
 is the Joseph ideal of $U(\frak{g})$.
We can regard $X$ as a $(\frak{g},K)$-module
 (sometimes referred to as a Harish-Chandra bimodule) and
 we have that ${\cal V}_\frak{g}(X)$
 is the closure of the minimal nilpotent
 $K_\bb{C}$-orbit in $\frak{p}^*_\bb{C}$.
Hence Theorem~\ref{geldd} shows that 
 $X$ is discretely decomposable as a $(\frak{g}^\sigma,K^\sigma)$-module.

By the classification \cite[Table 3 and Table 4]{KoOs}, 
(3) is satisfied 
 for the pairs in Table~\ref{ddpairs}
 with $\frak{g}=\frak{sl}(4,\bb{R})$, 
 $\frak{su}(2m,2n)$, $\frak{so}(m,n)$, $\frak{sp}(2,\bb{R})$, $\frak{sp}(m,n)$,
 $\frak{sl}(2n,\bb{C})$, $\frak{so}(2n,\bb{C})$, $\frak{f}_{4(4)}$,
 $\frak{f}_{4(-20)}$, $\frak{e}_{6(2)}$, $\frak{e}_{6(-14)}$,
 $\frak{e}_{7(-5)}$, $\frak{e}_{8(-24)}$.

The only remaining pairs that are not covered by (1), (2) and (3)
 are
 $(\frak{g},\frak{g}^\sigma)
=(\frak{sl}(2n,\bb{R}), \frak{sl}(n,\bb{C})\oplus\frak{u}(1))$
 and $(\frak{sl}(2n,\bb{R}), \frak{sp}(n,\bb{R}))$.
In this case, let $G=SL(2n,\bb{R})$ and $P$
 a maximal parabolic subgroup of $G$ with
 Levi part $L=S(GL(2n-1,\bb{R})\times GL(1,\bb{R}))$.
Let $X$ be the underlying $(\frak{g},K)$-module of 
 a degenerate principal series representation 
 of $G$ induced from a character of $P$.
Then it turns out that
 $\op{AS}_K(X)=\bb{R}_{> 0}(-w_0\beta)$ and hence 
 $X$ is discretely decomposable
 as a $(\frak{g}^\sigma, K^\sigma)$-module by the criterion
 given in Fact~\ref{suff}.

Thus we have found at least one
 discretely decomposable $(\frak{g},K)$-module for
 all the pairs $(\frak{g},\frak{g}^\sigma)$ in Table~\ref{ddpairs}.
This completes the proof of the theorem.
\end{proof}

\begin{rem}
\rm{
Concrete branching laws are given in
 \cite{KOP}
for the last two cases 
 in the proof above,
 that is, 
$(\frak{g},\frak{g}^\sigma)
=(\frak{sl}(2n,\bb{R}), \frak{sl}(n,\bb{C})\oplus\frak{u}(1))$
 and $(\frak{sl}(2n,\bb{R}), \frak{sp}(n,\bb{R}))$.  
}
\end{rem}

From the proof of Theorem~\ref{thm:ddpair},
 we can take $X$ in the condition (i) of Theorem~\ref{thm:ddpair}
 to be unitarizable.

\begin{cor}[unitarizable $X$]
\label{unitary}
In the setting of Theorem~\ref{thm:ddpair},
 the conditions ${\rm (i)}$, ${\rm (ii)}$, and ${\rm (iii)}$
 are also equivalent to ${\rm (i')}$.
\begin{enumerate}
\item[${\rm (i')}$]
 there exists an infinite-dimensional
 irreducible unitarizable $(\frak{g},K)$-module $X$
 (by replacing $G$ with a covering group of $G$
 if necessary)
 such that
 $X$ is discretely decomposable
 as a $(\frak{g}^\sigma,K^\sigma)$-module.
\end{enumerate}
\end{cor}

\begin{proof}
It is enough to see that the $(\frak{g},K)$-modules $X$
 in the proof of Theorem~\ref{thm:ddpair} can be taken
 to be unitarizable
 in all cases.
For highest weight $(\frak{g}, K)$-modules, 
 we can take (for example) holomorphic discrete series representations.
For minimal representations, see \cite{To}.
For $X=U(\frak{g})/J$ with $\frak{g}$ complex
 and $J$ the Joseph ideal, see
 \cite[\S 12.4]{Hua} and \cite{Vo81}.
For $A_\frak{q}(\lambda)$
 and degenerate principal series representations, 
 we use the fact that
 the Zuckerman derived functor preserves
 unitarity under a certain positivity condition
 and that the classical parabolic induction
 always preserves unitarity.
\end{proof}

We pin down a special case
 that ${\frak{g}}$ is a complex simple Lie algebra:
\begin{cor}
\label{complex}
Suppose that $\frak{g}$ is a complex simple Lie algebra
 and $\frak{g}^\sigma$ is a real form of $\frak{g}$.
We regard the pair $(\frak{g}, \frak{g}^\sigma)$ as a symmetric pair of real Lie algebras.
Then the following six conditions on $(\frak{g}, \frak{g}^\sigma)$ are equivalent.
\begin{enumerate}
\item[{\rm(i)}]
There exists an infinite-dimensional irreducible $(\frak{g},K)$-module $X$
 such that $X$
 is discretely decomposable as a $(\frak{g}^\sigma,K^\sigma)$-module.
\item[{\rm(ii)}] 
 $\op{pr}(K_\bb{C}\cdot\frak{p}^*_{\beta})\subset 
{\cal N}({\frak{p}_\bb{C}^\sigma}^*)$.
\item[{\rm(iii)}] 
 $\sigma\beta\neq -\beta$
 ($\beta$ is the highest non-compact root given in Definition~\ref{beta}).
\item[{\rm(iv)}] 
The minimal nilpotent orbit of $\frak{g}$ does not intersect with 
 the real form $\frak{g}^\sigma$.
\item[{\rm{(v)}}]
The minimal nilpotent orbit of $\frak{g}^\sigma_\bb{C}(\simeq\frak{g})$
 does not intersect with $\frak{p}^\sigma_{\bb{C}}$.  
\item[{\rm(vi)}] 
The real form $\frak{g}^\sigma$ of $\frak{g}$
 is compact, or is isomorphic to
 $\frak{su}^*(2n)$, $\frak{so}(n-1,1) (n\geq 5)$, 
 $\frak{sp}(m,n)$, $\frak{f}_{4(-20)}$, or $\frak{e}_{6(-26)}$.
\end{enumerate}
\end{cor}

\begin{rem}
\rm{
{(1)}\enspace
Corollary~\ref{complex} generalizes
 \cite[Theorem 8.1]{kob98ii}, 
 which dealt with split real forms 
 $\frak {g}^{\sigma}$ of $\frak {g}$.  
\newline
{(2)}\enspace
In the condition (vi) of Corollary~\ref{complex},
 the associated symmetric pair
 $(\frak {g}, \frak {g}^{\theta \sigma})$
 is 
 $(\frak {g}, \frak {g})$, 
 or is a complex symmetric pair 
 $(\frak {sl}(2n,\bb {C}), \frak {sp}(n,\bb{C}))$, 
 $(\frak {so}(n,\bb {C}), \frak {so}(n-1,\bb{C})) (n\geq 5)$, 
 $(\frak {sp}(m+n,\bb {C}), \frak {sp}(m,\bb{C})\oplus\frak{sp}(n,\bb{C}))$, 
 $(\frak {f}_{4}^\bb{C}, \frak {so}(9,\bb{C}))$, or
 $(\frak {e}_{6}^\bb{C}, \frak {f}_4^\bb{C})$, respectively.  
}
\end{rem}
\begin{proof}
The equivalence of (i), (ii), (iii) and (vi) follows from
 Theorem~\ref{thm:ddpair} and Lemma~\ref{projorbit}.
If $\frak{g}$ is a complex simple Lie algebra,
 then $\frak {k}$ is a real form of 
$\frak g$ and there is a natural 
isomorphism of complex Lie algebras
 $\iota:\frak{k}_\bb{C}\xrightarrow{\sim}\frak{g}$
 that is identity on $\frak{k}$.
Then 
 $\iota(\frak{k}^\sigma+\sqrt{-1}\frak{k}^{-\sigma})=\frak{g}^\sigma$.
Put $\frak{a}:=\iota(\sqrt{-1}\frak{t}^{-\sigma})$,
 $\frak{h}:=\iota(\frak{t}_\bb{C})$,
 and let $\Delta^+(\frak{g},\frak{h})$ be the positive system
 corresponding to $\Delta^+(\frak{k}_\bb{C},\frak{t}_\bb{C})$.
Then $\frak{h}$ is a Cartan subalgebra of $\frak{g}$, 
 $\frak{a}$ is a maximal abelian subalgebra of $\frak{p}^\sigma$,
 and $\Delta^+(\frak{g},\frak{h})$ is compatible
 with some positive system $\Sigma^+(\frak{g},\frak{a})$
 of the restricted root system.
Under the isomorphism $\iota$, 
 the $\frak{k}_\bb{C}$-module $\frak{p}^*_\bb{C}$
 can be identified with the adjoint representation of $\frak{g}$.
Then the weight $\beta$ corresponds to
 the highest root in $\Delta^+(\frak{g},\frak{h})$
 and the condition (iii) amounts to that the highest root is
 zero on $\frak{t}^\sigma$ (i.e.\ a real root).
By a result of T.~Okuda~\cite{Ok}, this is equivalent to (iv).
The equivalence of (iv) and (v) follows from
 the Kostant--Sekiguchi correspondence \cite[Proposition 1.11]{Se}.
\end{proof}

%%%%%%%%%%%%%%%%%%%%%%%%%%%%%%%%%%%%%%%%%%%%%%%%%%%%%%%

\section{Discretely Decomposable Tensor Product}
\label{sec:tensor}

The tensor product 
 of two irreducible representations
 is regarded as a special case of our setting.

Let $G$ be a connected simple Lie group.
Let $\frak{g}=\frak{k}+\frak{p}$ be a Cartan decomposition
 of the Lie algebra
 and $K$ the connected subgroup with Lie algebra $\frak{k}$. 
Put $\widetilde{G}=G\times G$, $\widetilde{K}=K\times K$ 
 and let $\sigma$ act on $\widetilde{G}$ by switching factors.
Then any irreducible $(\widetilde{\frak{g}}, \widetilde{K})$-module $X$
 is of the form of the exterior product $X_1\boxtimes X_2$
 with two irreducible $(\frak{g},K)$-modules
 $X_1$ and $X_2$.
Then $X$, 
 regarded as
 a $(\widetilde{\frak{g}}^\sigma, \widetilde{K}^\sigma)$-module
 by restriction,
 is nothing but the tensor product representation $X_1\otimes X_2$.
The following theorem determines
 when $X_1 \otimes X_2$ is discretely decomposable.  

\begin{thm}
\label{tensor}
Let $G$ be a  non-compact connected simple Lie group.
Let $X_1$ and $X_2$ be
 infinite-dimensional irreducible $(\frak{g}, K)$-modules.
Then the tensor product representation
 $X_1\otimes X_2$ is discretely decomposable as a $(\frak{g}, K)$-module
 if and only if $G$ is of Hermitian type
 and both $X_1$ and $X_2$ are simultaneously
 highest weight $(\frak{g},K)$-modules
 or simultaneously lowest weight $(\frak{g},K)$-modules.
\end{thm}

\begin{proof}
If $X_1$ and $X_2$ are both highest weight $(\frak{g},K)$-modules
 or they are both lowest weight $(\frak{g},K)$-modules, it is known that
 the tensor product $X_1\otimes X_2$
 is discretely decomposable
 (see \cite[Theorem 7.4]{kob98iii}).

Conversely, let us prove that
 $X_1\otimes X_2$ is not discretely decomposable
 as a $(\frak{g}, K)$-module unless $X_1$ and $X_2$
 are highest weight modules or they are lowest weight modules.
Let $\frak{t}$ be a Cartan subalgebra of $\frak{k}$.
Fix a positive system $\Delta^+(\frak{k}_\bb{C},\frak{t}_\bb{C})$.
We set $\widetilde{\frak{g}}=\frak{g} \oplus \frak{g}$, 
 $\widetilde{\frak{k}}=\frak{k} \oplus \frak{k}$, and
 $\widetilde{\frak{t}}=\frak{t} \oplus \frak{t}$.  
Then $\widetilde{\frak{t}}$
 is a Cartan subalgebra of $\widetilde{\frak{k}}$.  
We have an isomorphism $\widetilde{\frak{g}}^\sigma\simeq \frak{g}$ and 
 the restriction map
 $\op{pr}:\widetilde{\frak{g}}_\bb{C}^*
 \to {\widetilde{\frak{g}}}_\bb{C}^{\sigma*}$
 is identified with the map
 $\frak{g}_\bb{C}^*\oplus \frak{g}_\bb{C}^*\to \frak{g}_\bb{C}^*$
 given by $(x,y)\mapsto x+y$.
We take a positive system
 $\Delta^+(\widetilde{\frak{k}}_\bb{C},\widetilde{\frak{t}}_\bb{C})$ 
 to be the union of $\Delta^+(\frak{k}_\bb{C},\frak{t}_\bb{C})$ 
 in the first factor and $-\Delta^+(\frak{k}_\bb{C},\frak{t}_\bb{C})$ 
 in the second factor so that
 $\Delta^+(\widetilde{\frak{k}}_\bb{C},\widetilde{\frak{t}}_\bb{C})$
 is $(-\sigma)$-compatible.
Let $\beta\in\sqrt{-1}\frak{t}$ be
 the highest non-compact root given in Definition~\ref{beta}.

Suppose that $\frak{g}$ is not of Hermitian type.
Since $X_1$ and $X_2$ are infinite-dimensional, 
 we have ${\cal V}_{\frak{g}}(X_1), {\cal V}_{\frak{g}}(X_2)\neq \{0\}$
 by Lemma~\ref{lem:ass}(2).
Hence they contain $K_\bb{C}\cdot\frak{p}_\beta^*$
 and $K_\bb{C}\cdot\frak{p}_{-\beta}^*$, in particular 
 ${\cal V}_{\frak{g}}(X_1)\supset \frak{p}^*_{\beta}$
 and ${\cal V}_{\frak{g}}(X_2)\supset \frak{p}^*_{-\beta}$.
We therefore have 
\begin{align*}
{\cal V}_{\widetilde{\frak{g}}}(X_1\boxtimes X_2)
={\cal V}_{\frak{g}}(X_1)\oplus{\cal V}_{\frak{g}}(X_2)
\supset \frak{p}^*_{\beta}\oplus\frak{p}^*_{-\beta}.
\end{align*}
As in the proof of Lemma~\ref{projorbit}, 
 we can see 
$\op{pr}(\frak{p}^*_{\beta}\oplus\frak{p}^*_{-\beta})
 \not\subset {\cal N}(\frak{p}^*_\bb{C})$ and hence
$\op{pr}({\cal V}_{\widetilde{\frak{g}}}(X_1\boxtimes X_2))
\not\subset {\cal N}(\frak{p}^*_\bb{C})$.
Therefore Fact~\ref{nec} shows that
 $X_1\otimes X_2$ is not discretely decomposable.

Suppose that $\frak{g}$ is of Hermitian type.
By Lemma~\ref{lem:ass} (2) and Lemma~\ref{lem:hwas},
 if a highest weight $(\frak{g},K)$-module $X$ is
 also a lowest weight $(\frak{g},K)$-module,
 then $X$ is finite-dimensional.
Since $X_1$ and $X_2$ are infinite-dimensional,
 at least one of the following holds.
\begin{enumerate}
\item[(1)] $X_1$ and $X_2$ are highest weight modules.
\item[(2)] $X_1$ and $X_2$ are lowest weight modules.
\item[(3)] $X_1$ is not a lowest weight module
 and $X_2$ is not a highest weight module.
\item[(4)] $X_1$ is not a highest weight module
 and $X_2$ is not a lowest weight module.
\end{enumerate}
By switching $X_1$ and $X_2$ if necessary,
 it is enough to prove that $X_1\otimes X_2$
 is not discretely decomposable under
 the assumption (3).
We thus assume
 that $X_1$ is not a lowest weight $(\frak{g},K)$-module
 and $X_2$ is not a highest weight $(\frak{g},K)$-module.
By Lemma~\ref{lem:hwas}, 
 this assumption is equivalent to
 ${\cal V}_{\frak{g}}(X_1)\not\subset\frak{p}^*_-$
 and  ${\cal V}_{\frak{g}}(X_2)\not\subset\frak{p}^*_+$.
Hence it follows from the proof of Proposition~\ref{minorbit}
 that ${\cal V}_{\frak{g}}(X_1)\supset\frak{p}^*_{\beta}$
 and  ${\cal V}_{\frak{g}}(X_2)\supset\frak{p}^*_{-\beta}$.
Then by using the previous argument
 we see that $X_1\otimes X_2$ is not discretely decomposable.
\end{proof}

\end{document}